\documentclass[10pt]{article}

\usepackage{amssymb}
\usepackage[svgnames]{xcolor}
\usepackage{hyperref}
\usepackage{enumitem}
\usepackage{mathtools}
\usepackage{geometry}
\usepackage{amsfonts}
\usepackage{amsmath}
\usepackage{amsthm}
\usepackage{thmtools}
\usepackage{mathrsfs}
\usepackage{esint,subfig}
\usepackage{latexsym}
\usepackage[numbers]{natbib}
\usepackage{bm,bbm}
\usepackage{caption}
\usepackage{float}
\usepackage{multicol}
\usepackage{authblk}

\allowdisplaybreaks

\numberwithin{equation}{section}

\usepackage{etoolbox}
\patchcmd{\section}{\scshape}{\bfseries}{}{}

\hypersetup{colorlinks=true,allcolors=carminepink} 

\theoremstyle{plain}
\declaretheorem[parent=section]{theorem}
\declaretheorem[sibling=theorem]{proposition}

\declaretheorem[sibling=theorem,style=remark]{remark}

\declaretheorem[sibling=theorem,style=definition]{definition}
\declaretheorem[sibling=theorem,style=definition]{condition}
\DeclareSymbolFont{bbold}{U}{bbold}{m}{n}
\DeclareSymbolFontAlphabet{\mathbbold}{bbold}

\definecolor{bittersweet}{rgb}{1.0, 0.44, 0.37}
\definecolor{electricultramarine}{rgb}{0.05, 0.05, 1.0}
\definecolor{rred}{RGB}{152,0,0}
\definecolor{carminepink}{rgb}{0.92, 0.3, 0.26}

\renewcommand\labelenumi{(\roman{enumi})}
\renewcommand\theenumi\labelenumi

\newcommand*\diff{\mathop{}\!\mathrm{d}}
\newcommand{\Ex}{\mathbb{E}}				  
\newcommand{\Prob}{\mathbb{P}}

\newcommand{\ind}[1]{\mathbbold{1}{\bbra{#1}}} 

\newcommand{\argmin}{\operatorname{argmin}}
\newcommand{\qqqquad}{\quad \quad \quad \quad}

\providecommand{\abs}[1]{\left\vert#1\right\vert}			
\providecommand{\norm}[1]{\left\Vert#1\right\Vert}

\newcommand{\sphere}{\mathbb{S}^2}				         
\newcommand{\s}{\mathbb{S}^2}

\newcommand{\sumltrunc}{\sum_{\ell=0}^{\LL}}

\newcommand{\reals}{\mathbb{R}}					          
\newcommand{\integers}{\mathbb{Z}}
\newcommand{\naturals}{\mathbb{N}}

\newcommand{\Ltwo}{{L^2\bra{\sphere}}}

\newcommand{\Y}{Y_{\ell, m}}    
    
\newcommand{\ylm}{Y_{\ell, m}}

\newcommand{\alm}{a_{\ell, m}} 
 
\newcommand{\almZ}{a_{\ell, m;Z}}
\newcommand{\Cl}{C_{\ell}}
\newcommand{\ClZ}{C_{\ell;Z}}
\newcommand{\LegP}{P_{\ell}}

\renewcommand{\theta}{\vartheta}
\newcommand{\Hl}{\mathcal{H}_\ell}

\newcommand{\summ}{\sum_{m=-\ell}^{\ell}}     
\newcommand{\sumlm}{\sum_{\ell\geq 0}\sum_{m=-\ell}^{\ell}}

\newcommand{\primen}{\ell^{\prime}, m^{\prime}}

\newcommand{\Tup}{N}

\newcommand{\sumt}{\sum_{t=p+1}^{\nup}}

\newcommand{\Z}{Z\(x,t\)}

\providecommand{\bra}[1]{\left(#1\right)}			      
			    
\providecommand{\bbra}[1]{\{#1\}}			      
\renewcommand{\(}{\left(}
\renewcommand{\)}{\right)}
\renewcommand{\[}{\left[}
\renewcommand{\]}{\right]}
\renewcommand{\{}{\left\lbrace}
\renewcommand{\}}{\right\rbrace}

\newcommand{\lasso}{\operatorname{lasso}}

\newcommand{\phib}{\boldsymbol{\phi}}

\newcommand{\legp}{P_{\ell}}
\newcommand{\legf}{P_{\ell,m}}
\newcommand{\Phij}{\Phi_j}
\newcommand{\philj}{\phi_{\ell;j}}
\renewcommand{\theta}{\vartheta}
 
\newcommand{\sumjp}{\sum_{j=1}^{p}}
\newcommand{\Sl}{S_\ell ({\phib}_\ell)}

\newcommand{\kLASSO}{\widehat{\mathbf{k}}^{\lasso}_N}

\newcommand{\PpN}{\mathcal{P}^p_N}
\newcommand{\ka}{\mathbf{k}}
\newcommand{\sumtp}{\sum_{t=p+1}^n}
\providecommand{\TRnorm}[1]{\lVert #1\rVert_{\operatorname{TR}}}

\newcommand{\nup}{n}
\newcommand{\LN}{L_N}
\newcommand{\LL}{\LN-1}
\newcommand{\kvec}{\mathbf{k}}
\newcommand{\phiLASSO}{\widehat{\phib}_{\ell,N}^{\lasso}}
\newcommand{\phiLASSOpr}{\widehat{\phib}_{\ell^\prime,N}^{\lasso}}

\newcommand{\YlN}{\mathbf{Y}_{\ell,N}}
\newcommand{\YlNh}{\mathbf{Y}_{\ell,N}\(h\)}
\newcommand{\ElN}{\mathbf{E}_{\ell,N}}
\newcommand{\XlN}{X_{\ell,N}}
\newcommand{\GammalNest}{\widehat\Gamma_{\ell,N}}
\newcommand{\GammalN}{\Gamma_{\ell}}
\newcommand{\Mf}{\mathcal{M}(f_\ell)}
\newcommand{\mf}{\mathfrak{m}(f_\ell)}
\newcommand{\MM}{\mathcal{M}}

\newcommand{\mumax}{\mu_{\max; \ell}}
\newcommand{\mumin}{\mu_{\min; \ell}}
\newcommand{\gammalN}{\widehat\gamma_{\ell,N}}
\newcommand{\ql}{q_\ell}

\newcommand{\rr}{r}	

\def\R{\mathbb{R}}
\def\P{\mathbb{P}}
\def\C{\mathbb{C}}

\def\N{\mathbb{N}}

\def\Z{\mathbb{Z}}

\def\S{\mathscr{S}}

\def\to{\longrightarrow}
\def\sto{\rightarrow}

\def\1{\mathbbm{1}}

\begin{document}
	\bibliographystyle{alpha}

\author[1]{Alessia Caponera\footnote{E-mail: alessia.caponera@uniroma1.it. }}
\author[2]{Claudio Durastanti\footnote{Corresponding author. E-mail: claudio.durastanti@uniroma1.it. }}
\author[3]{Anna Vidotto\footnote{E-mail: vidotto@mat.uniroma2.it. }}
\affil[1]{Dipartimento di Scienze Statistiche, Sapienza Universit\`a di Roma}
\affil[2]{Dipartimento SBAI, Sapienza Universit\`a di Roma}
\affil[3]{Dipartimento di Matematica, Tor Vergata Universit\`a di Roma}


\title{LASSO estimation for spherical autoregressive processes}


\date{ \today}

\maketitle
\begin{abstract}
The purpose of the present paper is to investigate on a class of spherical functional autoregressive processes in order to introduce and study LASSO (Least Absolute Shrinkage and Selection Operator) type estimators for the corresponding autoregressive kernels, defined in the harmonic domain by means of their spectral decompositions. Some crucial properties for these estimators are proved, in particular, consistency and oracle inequalities.
\end{abstract}
\maketitle

\section{Introduction}\label{sec:intro}

\subsection{Background and motivations}\label{sub:background}
In recent years, growing attention has been paid to space-time processes built over various domains as, for example, regular grids, Euclidean spaces, and Riemannian manifolds (see, among others, \cite{JZ:97,gneiting,kulik,porcu,clarke,cm19,BINGHAM2019}). In particular, space-time random fields defined over the two-dimensional unit sphere $\s$ find a wide set of applications in Cosmology, Geophysics, and Medical Imaging, providing a tool to perform analysis of data evolving with time and distributed over the sphere representing, for instance, either the Universe (see \cite{MP:11}), the planet Earth (see \cite{christakos}) or the human brain (see \cite{tk}). Moreover, statistical techniques dealing with the analysis of -- both purely spatial and depending on time -- spherical random fields have received considerable attention, for example, in \cite{bkmpAoSb,mal,langschwab,fan} and in \cite{dmp,cm19}, respectively. 

In the last two decades, classes of space-time covariance functions has been defined and examined in the perspective to build the so-called sphere-cross-time random fields. The reader is referred to \cite{clarke, gneiting, jun, porcu, steinst} and the references therein for some neat examples. Another construction involving sphere-cross-time random fields has been presented in \cite{dmp} (see also \cite{bdmp}), where quantitative central limit theorems for linear and non-linear statistics based on spherical time-dependent Poisson random fields have been established.\\

A sphere-cross-time random field is denoted by a collection of random variables 
$$\{T\( x,t\):\(x,t\) \in \s\times\integers\},$$
and it can be described in the so-called frequency domain by its harmonic expansion
\begin{equation}\label{eq:randomfield}
T\(x,t\)= \sumlm \alm\(t\) \ylm\(x\), \quad \(x,t\)\in\s \times \integers \,,
\end{equation}
where $\{\ylm:\ell \geq 0; m=-\ell,\ldots,\ell\}$ is the standard orthonormal basis of real spherical harmonics for $\Ltwo=L^2\(\s,\diff x\)$, the space of square-integrable functions on the sphere with respect to the spherical Lebesgue measure $\diff x$ (see Section \ref{sec:setting} and the references \cite{steinweiss,yadrenko}). Harmonic analysis has already been proved to be a valid tool to perform statistical analysis on the sphere (see, for instance, \cite{dlm,kimkoo}) and allows one to describe a random field as the linear combination of spherical harmonics, weighted by the corresponding time-varying harmonic coefficients (see \cite{BINGHAM2019}). The stochastic information of the random field $T$ is then contained in the set of harmonic coefficients $\{\alm\(t\):\ell \geq 0; m=-\ell,\ldots,\ell\}$, given by 
\begin{equation*}\label{eq:alm}
\alm\(t\) = \langle T\(\cdot,t\), \ylm \rangle _{\s}\,,
\end{equation*}
where $\langle \cdot, \cdot \rangle _{\s}$ is the standard inner product over the sphere.\\

In this paper, the object of study is a class of sphere-cross-time random fields, introduced in the literature by \cite{cm19}, which are functional autoregressive processes defined over $\Ltwo$ (see also \cite{bosq} and the references therein). As the name suggests, the spherical autoregressive model of order $p$, from now on shortened to $\operatorname{SPHAR}(p)$, specifies the output field $T(\cdot, t)$ as an infinite-dimensional linear transformation of its $p$ previous realizations ($p$ lags) added to an independent \emph{spherical white noise} $Z\(\cdot,t\)$. More formally, the 
$\operatorname{SPHAR}(p)$ equation is given by 
\begin{equation}\label{eq:arp}
T\(x,t\)= \sum_{j=1}^{p} \(\Phij T\(\cdot,t-j\)\)\(x\) + Z\(x,t\)\,, \qquad \(x,t\) \in \s\times\Z \,,
\end{equation}
where, for $j=1,\ldots,p$, the autoregressive kernel operator $\Phij:\Ltwo\rightarrow\Ltwo$ is defined as  
\begin{equation*}
\(\Phij f\) \(x\) = \int_{\s} k_j \(\langle x , y \rangle\) f\(y\) \diff y, \quad f \in \Ltwo,
\end{equation*}
with $k_j: \[-1,1\]\rightarrow \reals$ the corresponding autoregressive kernel, assumed to be continuous.  
Then, for $j=1,\ldots, p$, the estimation of $\Phij$ can be reduced to the one of $k_j$. Note that, for any $j = 1,\ldots, p$, the kernel $k_j$ is isotropic, that is, it depends only on $z=\langle x,y\rangle $, the standard inner product on $\R^3$. As a consequence, the following spectral representation holds in the $L^2$-sense
\begin{align}\label{eq:kj-dec}
k_j\(z\)=& \sum_{\ell\geq 0} \philj \frac{2\ell+1}{4\pi}\legp \(z\), 
\end{align} 
where $\legp:\[-1,1\]\rightarrow \reals$ is the Legendre polynomial of order $\ell$, and, for $j=1,\ldots,p$, the coefficients $\{\philj:\ell \geq 0\}$ are
the eigenvalues of the operator $\Phij$. 
The reader is referred to Section \ref{sec:setting} for a detailed discussion. Moreover, in line with \cite{cm19}, the random field $Z$ in \eqref{eq:arp} is defined such that, for any $x,y \in \s$, the covariance function $\Gamma_Z\(x,y\)=\Ex\[Z\(x,t\)Z\(y,t\)\]$ is isotropic and independent on the choice of $t \in \integers$ (see \cite{yadrenko}).\\
As a standard consequence of the so-called \emph{duplication property} for spherical harmonics, 
we can characterize functional spherical autoregressions in the frequency domain as 
\begin{equation}\label{eq:auteq}
\alm\(t\) = \sumjp \philj \alm \(t-j\) + \almZ \(t\).
\end{equation}     
In \cite{cm19}, estimators for the kernels $\{k_j: j=1,\ldots,p\}$ have been defined according to a functional  $L^2$-minimization criterion, exploiting their spectral decomposition \eqref{eq:kj-dec}; some asymptotic properties, such as consistency, quantitative central limit theorem, and weak convergence, have then been established under additional regularity conditions. \\

In this work, we approach the estimation of the autoregressive kernels under a \emph{sparsity} assumption, that is, for any $j=1,\ldots,p$, we assume that only a few of the components of $k_j$ in \eqref{eq:kj-dec} are non-zero (see Definition \ref{cond:sparsity} below). Assuming sparsity conditions can lead to considerable advantages in estimation problems (see, for example, \cite{tibshibook1}). In this case, the proper identification of the null components allows to perform regularization on the functional estimates, preserving accuracy; moreover, sparsity enhances computational efficiency. 

LASSO - or $\ell_1$-regularized - regression, introduced in the statistical literature by the celebrated paper \cite{tibshi}, is among the most popular penalization techniques to estimate sparse models. In particular, it corrects the $L^2$-loss for sparse models by adding a convex penalty term and, then, constraining the estimation process and selecting the most significant variables. 
In the framework of independent and identically distributed ($\text{i.i.d.}$) observations, LASSO has been proved to be extremely efficient both from the point of view of theoretical properties and in terms of applications  (see \cite{tibshibook2,wainwright2019} and references therein).
The connections between LASSO, ridge regression, best subset selection and other $\ell_q$-based penalization methods, as well as further links between LASSO and other nonparametric statistical techniques, such as soft and hard thresholding, have been widely investigated, for instance, in \cite{bvdg,tibshibook1}. 

Applications of LASSO in the framework of time series and stochastic processes represent a much more recent development. A pioneering contribution in this area has been given in \cite{BM:15}, where the authors explore the properties of $\ell_1$-regularized estimators in the settings of stochastic regression with serially correlated errors and vector autoregressive (VAR) models (see also \cite{daviszangzheng,songbickel} for related ideas). 
Their results can be seen as a successful extension of the standard LASSO technique to the framework of non-$\text{i.i.d.}$ observations. 
More specifically, in \cite{BM:15}, under sparsity constraints, $\ell_1$-regularized estimators have been investigated by introducing a measure of stability for stationary processes, a very powerful tool to study the correlation structure of multivariate processes, and crucial to settle some useful deviation bounds for dependent data. 
%
%
Further details on the stability of autoregressions can be found, among others, in \cite{brockwelldavis,stablelutti} as well as in \cite{bosq} for the functional case (see also Section \ref{sec:LASSO}). In turn, these deviation bounds are instrumental to establish concentration properties of the estimators, and so-called \emph{oracle inequalities}.\\

The aim of this work is to define and study LASSO-type estimators for spherical autoregressive kernels under sparsity assumptions. 
In line with \cite{cm19}, our approach does not require any specific functional form for the kernel $k_j$; in this sense, the estimation procedure can be viewed as fully nonparametric, see also \cite{prakasa,tsybakov}. It is important to stress that, given the nonparametric nature of the model \eqref{eq:arp}, we are dealing with a \textit{functional} penalized regression problem, hence differing from the framework of $\operatorname{VAR}(p)$ processes, where estimators assume a vectorial form (see also Remark \ref{rmk:highdim}). More specifically, our oracle inequalities will involve functions rather than scalar or vectorial parameters (see Section \ref{sub:fmr}).
Exploiting the harmonic expansion for the spherical autoregressions \eqref{eq:arp} and the isotropy assumption on $\{k_j : j = 1, \ldots, p\}$ in \eqref{eq:kj-dec}, together with an extension of the concept of stability measure introduced in \cite{BM:15}, we will be able to establish concentration properties in functional norms for the autoregressive kernels (see Section \ref{sec:MR}). Moreover, the sparsity enforcement properties of LASSO procedures will avoid overfitting and will select only the most relevant components of each kernel function spectral decomposition \eqref{eq:kj-dec} .

\subsection{Background and main results}\label{sub:fmr}
First of all, let us set some standard notation, necessary in order to state our main findings. For two real valued sequences $\{a_n\}_{n \in \mathbb{N}}, \{b_n\}_{n \in \mathbb{N}}$, we write $a_n \succeq b_n$ if
there exists an absolute constant $c$, which does not depend on the model parameters, such that $a_n \geq c\, b_n$, for all $n\in \N$. 
For a vector $v \in \mathbb{R}^d$, $\norm{v}_{q}$ denotes the $\ell_q$-norm of $v$,  $$\norm{v}_{q}=\left(\sum_{i=1}^d\abs{v_i}^q\right)^{\frac{1}{q}}; \quad \norm{v}_{0}=\sum_{i=1}^d \ind{v_i\neq 0}; \quad \norm{v}_{\infty} = \max_{i=1,\ldots,d} \abs{v_i},$$
for $0<q<\infty$, $q=0$ and $q= \infty$ respectively. We say that $v$ is a $r$-sparse vector, $1\le r \le d$, if $\norm{v}_{0}=r$.
Unless stated otherwise, for the sake of simplicity, $\norm{\cdot}$ denotes the $\ell_2$-norm of $v$. Let $f:\[-1,1\]\rightarrow \reals^p$, $f\in L^q \left(\[-1,1\], \rho\(\diff z\)\right)$, where $\rho\(\diff z\)$ is the Lebesgue measure over $\[-1,1\]$. Then, for $1\le q< \infty$, the $L^q$-norm of $f$ is given by $$\norm{f}_{L^q} = \(\int_{-1}^{1} \norm{f\(z\)}^q\rho\(\diff z\)\)^{\frac{1}{q}}.$$ Analogously, the $L^\infty$-norm of $f$ is given by $\norm{f}_{L^\infty}= \sup_{z \in\[-1,1\]}\norm{f\(z\)}$. The spherical $L^q$-norms are defined by $$\norm{f}_{L^q\(\s\)} = \(\int_{\s} \abs{f\(x\)}^p \diff x\)^{\frac{1}{q}}, \quad f \in L^q\(\s\).$$
Finally, the Hilbert-Schmidt and the trace class norms of a compact self-adjoint operator $\mathscr{T}:\mathbb{H}\rightarrow \mathbb{H}$, where $\mathbb{H}$ is a separable Hilbert space, are given respectively by 
$$
\norm{\mathscr{T}}_{\operatorname{HS}} = \sum_{i \geq 0} \abs{\lambda_i}^2; \quad \TRnorm{\mathscr{T}} = \sum_{i \geq 0} \abs{\lambda_i},
$$
where $\{\lambda_i\}_{i \in \naturals}$ are the eigenvalues of $\mathscr{T}$ (see, for example, \cite{Hsing}).  \\

We provide now the definition of sparsity set, which can be understood as an instrumental characterization of sparsity. 
\begin{definition}[Sparsity set]\label{cond:sparsity}
	For any $\ell \geq 0$ and $\phib_\ell = \(\phi_{\ell;1},\ldots,\phi_{\ell;p}\)$, we define $\ql = \norm{\phib_\ell}_0$ the $\ell$-th sparsity index, which satisfies $0 \le q_\ell \le p$. We call $\{q_\ell: \ell \ge 0\}$ the sparsity set.
\end{definition}
\begin{remark}\label{rem:remarkability}
	Following \cite{bosq,cm19}, to ensure identifiability we assume that there exists at least one $\ell \geq 0$
	such that $\phi _{\ell ;p}\neq 0$, so that $\P\(\Phi _{p}T\(\cdot,t\)\neq 0\)>0$, for all $t\in \mathbb{Z}$. As a consequence, for some $\ell\ge 0$, we can have $\phib_\ell = 0$ and hence $\norm{\phib_\ell}_0  = q_\ell = 0$; however, $q = \max_{\ell \ge 0} q_\ell \ge 1$. Note that, in the simplest case $p=1$, the sparsity set gives information about the null components of the single kernel.
\end{remark}

Set the sequence of polynomials $\phi _{\ell }:\mathbb{C\rightarrow C}$, $\ell \geq 0$, associated with the subprocesses defined in \eqref{eq:auteq}, as follows
\begin{equation}
\phi_{\ell}(z)= 1- \phi _{\ell;1}z-\cdots -\phi _{\ell ;p}z^{p}.
\label{AssoPoly}
\end{equation}  

We must introduce some conditions on the model, essential to achieve our results. 
As stated also in \cite{cm19}, Condition \ref{cond:identify} will require that the eigenvalues of the covariance operator associated with $\Gamma_Z(\cdot, \cdot)$ are positive (see \cite[Section 7.3]{Hsing}).
Condition \ref{cond:stazionarity} will guarantee the existence of a unique isotropic and stationary solution for \eqref{eq:arp}. 
For rigorous proofs, the reader is referred to \cite{bosq, Cathesis}. 

\begin{condition}[Identifiability]\label{cond:identify}
	Let $Z\(x,t\)$ be the spherical white noise used in \eqref{eq:arp}. It holds that 
	\begin{equation*}\label{eq:identify}
	\int_{\s\times \s}\Gamma_Z \(x,y\)f\(x\)f\(y\) \diff x \diff y>0,  
	\end{equation*} 
	for any $f \in \Ltwo$ such that $f\(\cdot\)\neq 0$.
\end{condition}

\begin{condition}[Stationarity]
\label{cond:stazionarity} The sequence of polynomials \eqref{AssoPoly} is such
that $$|z| \le 1 \ \Rightarrow \ \phi _{\ell }(z)\ne 0.$$ More explicitly,
there are no roots in the unit disk, for all $\ell \geq 0$.
\end{condition}

\begin{remark}\label{rmk:infinf}
Condition \ref{cond:stazionarity}, together with the spectral decomposition \eqref{eq:kj-dec} for the kernels $\{k_j, j=1,\dots,p\}$, implies that, if $\xi_{\ell;1},\dots,\xi_{\ell; d_\ell}$ are the roots of the $d_\ell$-degree
polynomial \eqref{AssoPoly}, $1 \le d_\ell \le p$, then
\begin{equation*}
|\xi_{\ell;j}| \ge \xi_\ast > 1,
\end{equation*}
uniformly over $\ell$. In other words, there exists $\delta >0$ such that $$|z| < 1 + \delta \ \Rightarrow \ \phi _{\ell }(z)\ne 0,\qquad \text{for all } \ell \geq 0.$$
\end{remark}

\begin{condition}[Smoothness]\label{cond:smooth}
For all $j=1,\dots,p$, we have that
\begin{equation}
\| \Phi_j \|_{\operatorname{TR}} = \sum_{\ell=0}^\infty (2\ell+1)|\phi_{\ell;j}| < \infty, 
\end{equation}
that is, $\Phi_j$ is a nuclear operator, see again \cite{Hsing}.
\end{condition}


As experimental setting, for any $\ell \geq 0$, we assume that the harmonic coefficients $\{\alm\(t\): m=-\ell,\ldots,\ell\}$ can be observed over a finite set of times $\{1,\ldots,n\}\subset\integers$. The vector of functions
$$\ka=\(k_1,\ldots,k_p\)$$ 
contains all the autoregressive kernels described above. 
We will focus on the following penalized minimization problem:
\begin{equation}\label{eq:aiuto}
\kLASSO =  \underset{\ka \in \PpN}{\argmin}  \, \frac1N \sumtp \norm{ T(\cdot, t) - \sumjp \Phij(T(\cdot, t-j)) }^2_{\Ltwo} +\lambda \sumjp\TRnorm{\Phij},
\end{equation}
where $N=n-p$ can be read as the effective number of observations, and $\lambda \in \reals^+$ is the penalty parameter. As well as in \cite{cm19}, the space $\PpN$ is the Cartesian product of $p$ copies of 
\begin{equation}\label{eq:PpN}
\operatorname{span}\{\frac{2\ell+1}{4\pi}\legp\(\cdot\): \ell=0,\ldots,L_N-1\}\,,
\end{equation}
where the integer $L_N>0$ is the truncation level, which corresponds to the frequency of the highest component in \eqref{eq:kj-dec} estimated by \eqref{eq:aiuto}, see Section \ref{lasso_est} for a detailed discussion. Let us also define
\begin{equation*}
k_{j, N}(z)=\sum_{\ell=0}^{\LL} \phi_{\ell;j} \frac{2\ell+1}{4\pi} P_\ell(z)
\end{equation*}
and, accordingly, $\mathbf{k}_N=(k_{1, N},\dots,k_{p, N})$. 
\\

Our main result is extensively stated in Theorem \ref{th:MR-preciso} and can be compactly formulated as follows. 
\begin{theorem}\label{th:summarize} Consider the estimation problem \eqref{eq:aiuto}, assume that Conditions \ref{cond:identify} and \ref{cond:stazionarity} hold, and suppose that
	\begin{equation*}\label{eq:DB}
	N \succeq b_1 \, q \log \(p \LN\),
	\end{equation*}
	where $q=\max_{\ell \ge 0} q_\ell$.
	Then, for any penalty parameter $\lambda = \lambda_N   \ge b_2 \sqrt{\frac{\log \(p \LN\)}{N}}$, the solution $\kLASSO$ of \eqref{eq:aiuto} satisfies 
	\begin{align}\label{eq:th1weak}
	&\Prob \(  \norm{\kLASSO - \ka }^2_{L^2} \le \frac{18}{\pi^2}\, \lambda_{N}^2\,\sumltrunc \frac{q_\ell}{\alpha_\ell^2} \,(2\ell+1)+\Big \|\ka - \ka_{N} \Big \|^2_{L^2} \)\ge 1-c_1e^{- c_2 \log \(p \LN\)} ,
	\end{align}
	where the sequence $\{\alpha_\ell:\ell=0,\ldots,\LL\}$ is defined in \eqref{anna}, and $c_1,c_2 >0$ are absolute constants. Moreover, under the additional Condition \ref{cond:smooth}, it holds that
	\begin{equation}
	\Prob \( \norm{\kLASSO - \ka }_{L^\infty} \le \frac{3}{\pi}\, \lambda_{N} \sumltrunc \frac{\sqrt{q_\ell}}{\alpha_\ell} \(2\ell+1\)+\Big \| \ka - \ka_{N} \Big \|_{L^\infty} \)
	\label{eq:th1strong}
	\ge 1-c_1 e^{-c_2 \log \(p \LN\)}.
	\end{equation} 
\end{theorem}

\begin{remark}\label{remark:constants}
The constants $b_1$ and $b_2$ depend on the model. In particular, $b_1=\max\{\omega^2,1\}$ and $b_2=4\mathcal{F}$, where $\mathcal{F}, \omega >0$ are tightly connected to the stability measure introduced in Section \ref{sub:stameasure}.
The reader is referred to Section \ref{sub:bounds} for further details and comments.
\end{remark}

Our findings provide upper bounds for the $L^2$- and the $L^\infty$-distances between the LASSO-type estimator $\kLASSO$ and the ``true'' $\mathbf{k}$. These upper bounds consist of the sum of two terms. The first summand represents the error due to the approximation of the first $\LN$ components of $\mathbf{k}$ with $\kLASSO$. The second one arises because  $\kLASSO$ provides an estimation of $\mathbf{k}$ truncated at the multipole $\LN$. In this sense, we can draw an analogy with standard nonparametric statistics and refer to them as the stochastic and the bias error, respectively (see \cite{prakasa}).   

The upper bounds are non-asymptotic and they hold with high-probability, in the sense that, for a fixed $N$ sufficiently large, the probability on the left side of \eqref{eq:th1weak} and \eqref{eq:th1strong} is arbitrarily close to $1$. An appropriate choice of $\LN$ leads both the upper bounds to converge to $0$ and their probabilities to converge to $1$, as $N \sto \infty$; as a consequence, our result can be also read in terms of asymptotic consistency. 

\subsection{Plan of the paper}\label{sub:plan}
This paper is organized as follows.  In Section \ref{sec:LASSO}, we present the LASSO estimators for spherical autoregressive kernels under sparsity assumptions. Section \ref{sec:MR} contains the main results of this work, that is, how the classical LASSO-scheme fits in our setting, using the concept of stability measure, as well as our oracle inequalities. In Section \ref{sec:simulations} we briefly show the performance of the LASSO estimators under sparsity assumptions. Finally, Section \ref{sec:proofs} collects the proofs. 

\section{General setting}\label{sec:setting}
\subsection{Harmonic analysis on the sphere and space-time spherical random fields}
This Section includes some well-established results concerning harmonic analysis on the sphere and provides the reader with an overview of the construction of space-time spherical random fields.\\
The reader is referred to \cite{ steinweiss,vilenkin,MP:11,langschwab} and the references therein for further details concerning harmonic analysis on the sphere and spherical random fields. Sphere-cross-time random fields have been discussed, among others, in \cite{gneiting, jun, porcu, steinst}.\\

Let us start by setting some useful notation and some important concepts related to the harmonic analysis on the sphere. 
Each point $x$ on the sphere is identified by two angular coordinates, that is, $x=\(\theta,\varphi\)$, where $\theta \in \[0,\pi\]$ and $\varphi \in \[\left.0,2\pi \)\right.$ are the colatitude and the longitude, respectively. The spherical Lebesgue measure is labeled by $\diff x=\sin\theta \diff\theta \diff\varphi$, while $\Ltwo=L^2\(\s,\diff x \)$ describes the space of square-integrable functions over the sphere with respect to the measure $\diff x$.\\
Following for example \cite{MP:11, steinweiss, vilenkin}, we denote by $\{\Hl: \ell\geq 0\}$ the set of spaces of homogeneouos harmonic polynomials of degree $\ell$ restricted to $\s$. For any $\ell\geq 0$, $\Hl$ is spanned by the set of spherical harmonics $\{\ylm :m=-\ell,\ldots,\ell \}$. The index $\ell \in \naturals$ is the so-called multipole, while $m=-\ell, \ldots,\ell$ is the ``azimuth'' number. The set $\{\Hl:\ell \geq 0\}$ is dense in $\Ltwo$ (see again \cite[Proposition 3.33, p.73]{MP:11}), thus the following decomposition holds 
\begin{equation*}
\Ltwo=\bigoplus_{\ell\geq 0}\Hl 
\end{equation*}
and spherical harmonics provide an orthonormal basis for $\Ltwo$. For the sake of simplicity, here we will make use of the so-called real spherical harmonics. For any $\ell \in \naturals$ and $m =-\ell,\ldots, \ell$, the spherical harmonic $\ylm$ can be written as the normalized product of the Legendre associated function $\legf: \[-1,1\]\rightarrow \reals$ of degree $\ell$ and order $m$, which depends only on the colatitude $\theta$, and a trigonometric function depending only on the longitude $\varphi$, namely,
\begin{align*}
\ylm \(\theta, \varphi\)= \{ \begin{matrix*}[l]
\sqrt{\frac{\(2\ell +1\)}{2 \pi} \frac{\(\ell-m\)!}{\(\ell+m\)!}} \legf \(\cos \theta \)\cos\(m \varphi\) &\text{for } m \in \{1,\ldots, \ell \}\\
\sqrt{\frac{\(2\ell +1\)}{4 \pi}} \legp  \(\cos \theta \) &\text{for } m = 0 \\
\sqrt{\frac{ \(2\ell +1\)}{2 \pi} \frac{\(\ell+m\)!}{\(\ell-m\)!}} P_{\ell,-m} \(\cos \theta \)\sin\(-m\varphi\) &\text{for } m \in \{-\ell,\ldots, -1 \}
\end{matrix*} \right.,
\end{align*}
where
\begin{equation*}
\legf \(u\)=\frac{1}{2^\ell \ell !} \(1-u^2\)^{\frac{m}{2}}\frac{\diff^{\ell + m }}{\diff u^{\ell+m}} \(u^2-1\)^{\ell}, \quad u \in \[-1,1\]. 
\end{equation*}
It is a well-known fact that the following addition formula holds
\begin{equation*}\label{eq:addition}
\summ \ylm\(x\)\ylm\(y\)=\frac{2\ell+1}{4\pi}\legp \(\langle x,y\rangle\), \quad x,y\in \s,
\end{equation*}
where $\legp$ is the Legendre polynomial of order $\ell$, given by
$$
\legp\(u\)=\frac{1}{2^{\ell}\ell !} \frac{\diff ^ \ell}{\diff u ^ \ell} \(u^2-1\)^\ell, \quad u \in \[-1,1\].
$$
Moreover, Legendre polynomials are orthogonal on $L^2([-1,1])$, that is,
\begin{equation}\label{eq:duplication}
\int_{-1}^1\legp \(u\)P_{\ell^\prime} \(u\) \, \diff u=\frac{2}{2\ell+1}\delta_{\ell}^{\ell^\prime}\,.
\end{equation}
\\

Recall that we are considering a sphere-cross-time random fields, which is a real-valued collection of random variables $$\{T\(x,t\):\(x,t\)\in \s \times \integers \}.$$ 
\begin{remark} All the results presented in this paper can be easily extended to the case of complex-valued spherical random fields, after a proper definition of complex spherical harmonics (see, for example, \cite[Theorem 5.13, p.123]{MP:11}).
\end{remark}
From now on, we will consider real-valued, centered, mean-square continuous, Gaussian random fields.  
Moreover, $T$ is assumed to be isotropic in the spatial domain and stationary in the time domain. A spherical random field is said to be isotropic when it is invariant in distribution with
respect to rotations, while stationarity guarantees that the stochastic properties of the process do not change over time. In other words
\begin{equation}\label{eq:isostat}
T\(R\,\cdot, \cdot +\tau\) \overset{d}{=} 	T\(\cdot, \cdot\), 
\end{equation}
where $\tau \in \integers$, $R$ belongs to the special group of rotations $SO\(3\)$, and $\overset{d}{=}$ denotes equality in distribution. 

Under isotropy, for any fixed $t \in \integers$, the following harmonic expansion holds
\begin{equation*}\label{eq:harmexptime}
T\(x,t\) = \sumlm \alm\(t\)\ylm \(x\), \quad \(x,t\)\in \s \times \integers\,.  
\end{equation*}
The set of the harmonic coefficients $\{\alm:\ell \geq 0, m=-\ell ,\dots,\ell\}$ contains all the stochastic information related to $T$ and can be computed explicitly by 
\begin{equation*}
\alm\(t\)=\int_{\s}T\( x,t\) \ylm \( x\) \diff x\,.
\end{equation*}%
Since $T$ is centered, that is, $\Ex\[T\(x,t\)\]=0$ for all $\(x,t\) \in \s \times \integers$, it follows that 
$$
\Ex\[\alm\(t\)\]= 0 \quad \text{for }\ell \in \naturals, \quad m = -\ell,\ldots, \ell, \quad t\in \integers\,.
$$
The covariance function of $T$ will be denoted by $\Gamma:\(\s\times \mathbb{Z}\)\times \(\s\times \mathbb{Z}\) \rightarrow \mathbb{R}$. If the space-time spherical random field is isotropic and stationary, then there exists a function $\Gamma_0: \[-1,1\]\times \mathbb{Z} \rightarrow \mathbb{R}$, so that \eqref{eq:isostat} yields
\begin{equation*}
\Gamma\(x,t,y,s\) = \Gamma_0\(\langle x,y\rangle, t-s \), \quad \(x,t\),\(y,s\) \in \s\times \mathbb{Z}\,.
\end{equation*}
Furthermore, for any $\ell, \ell^\prime \in \naturals$, $m=-\ell,\ldots,\ell$, $m^\prime= -\ell^\prime, \ldots, \ell^\prime$, the elements of the covariance matrix constructed over the harmonic coefficients of $T$ are given by
\begin{equation}\label{eq:cltime}
\Ex\[\alm\(t\) {a}_{\primen}\(s\)\] = \Cl\(t-s\)\delta_{\ell}^{\ell^\prime}\delta_{m}^{m^\prime}, \quad t,s \in \integers\,.
\end{equation}
Observe that, for $t=s$, $\Cl\(0\)$ in \eqref{eq:cltime} corresponds to the so-called angular power spectrum, the spectral decomposition of the covariance function of a purely spatial spherical random field (see, for example, \cite[Remark 5.15, p.124; Remark 6.16, p.147]{MP:11}). Hence, from now on, we will use the notation $\Cl\(0\)=\Cl$.

Covariance functions of isotropic-stationary sphere-cross-time random fields have a spectral decomposition in terms of Legendre polynomials, namely,
\begin{equation*}\label{eq:specdec}
\Gamma\(x,t,y,s\) = \sum_{\ell \geq 0} \Cl \(t-s\) \frac{2\ell+1}{4\pi} P_\ell \(\langle x,y \rangle \), \quad \(x,t\),\(y,s\)\in \s \times \integers\,,
\end{equation*}
as a consequence of Schoenberg Theorem (see also \cite{bergporcu}).
\subsection{Spherical autoregression}
Following \cite{cm19}, we consider isotropic-stationary random fields $\{T(x,t) \, (x,t) \in \s \times \mathbb{Z}\}$ that satisfy the functional autoregressive equation (see also \cite{bosq})
\begin{equation}\label{eq:sphar}
T\(x,t\)= \sum_{j=1}^{p} \(\Phij T\(\cdot,t-j\)\)\(x\) + Z\(x,t\)\,, 
\end{equation}
where:
\begin{itemize}
	\item for $j=0,\ldots,p$, $\Phij:\Ltwo\rightarrow\Ltwo$ is a linear and bounded operator defined by
	\begin{equation*}
	\(\Phij f\) \(x\) = \int_{\s} K_j \(x,y\) f\(y\) \diff y, \quad f \in \Ltwo.
	\end{equation*}
	\item The kernel $K_j:\s \times \s \rightarrow \reals $ is isotropic, that is, there exists a function $k_{j}:\[-1,1\]\rightarrow \reals$ so that $K_j \(x,y\)=k_{j}\(\langle x,y \rangle\)$, for all $x,y \in \s$. 
	Moreover, the following decomposition holds (in the $L^2$-sense and pointwise under Condition \ref{cond:smooth})
	\begin{align}
	\notag K_j\(x,y\)=k_{j}\(\langle x,y \rangle\)=& \sumlm \philj\ylm\(x\)\ylm\(y\)\\
	=& \sum_{\ell\geq 0} \philj \frac{2\ell+1}{4\pi}\legp \(\langle x,y \rangle \)\label{eq:kernel},
	\end{align}
	where the coefficients $\{\philj:\ell \geq 0\}$ are the eigenvalues of the operator $\Phij$ (see \cite{cm19}). 
	\item The field $Z$ is a Gaussian spherical white noise, that is, $\{Z\(\cdot, t\), \, t\in\Z\}$ is a sequence
	of independent and identically distributed Gaussian isotropic spherical random fields, defined so that 
	\begin{enumerate}
		\item for every fixed $t \in \Z$, $Z\(\cdot, t\)$ is a Gaussian, zero-mean isotropic random field, with
		covariance function given by
		\begin{equation*}
		\Gamma_Z\(x, y\) = \sum_{\ell \geq 0} \frac{2\ell+1}{4\pi}\ClZ \legp\(\langle x,y\rangle\), \quad \text{such that } \sum_{\ell \geq 0} \frac{2\ell+1}{4\pi}\ClZ<\infty,
		\end{equation*}
		where $\{\ClZ: \ell \geq 0 \}$ is the angular power spectrum of $Z\(\cdot, t\)$ such that $\ClZ>0$ for any $\ell \geq 0$ (cf. Condition \ref{cond:identify});
		\item  for every $t \neq s$, the random fields $Z\(\cdot, t\)$ and $Z\(\cdot, s\)$ are independent, so that, for any $x,y \in \s$, 
		\begin{equation*}
		\Ex \[Z\(x,t\){Z}\(y,s\)\]=0. 
		\end{equation*}
	\end{enumerate}
\end{itemize}

For any $t \in \integers$, \eqref{eq:kernel} yields 
\begin{equation*}
\(\Phij T\(\cdot,t-j\)\)\(x\)= \sumlm \philj \alm\(t-j\)\ylm\(x\), \quad x \in \s,  
\end{equation*}
namely, the spectral representation of $\(\Phij T\(\cdot,t-j\)\)\(\cdot\)$ is characterized by the following set of harmonic coefficients $\{\philj\alm\(t-j\):\ell \geq 0; m=-\ell, \ldots,\ell\}$. Consequently, from \eqref{eq:randomfield}, it follows that
\begin{equation}\label{eq:arcoeff}
\alm\(t\) = \sumjp \philj \alm \(t-j\) + \almZ \(t\),
\end{equation} 
that is, that is, the coefficients $\{\alm\(t\):t\in\integers\}$ follows an autoregressive model of order $p$, for any $\ell \geq 0$, $m=-\ell,\ldots, \ell$. \\

Before concluding this section, let us fix the so-called truncation frequency $L \in \naturals$. Then, the truncated random field is defined by 
\begin{equation*}
T_{L}\(x,t\)=\sum_{\ell=0}^{L-1} \summ \alm\(t\)\ylm\(x\). 
\end{equation*}
\begin{remark}
	The truncated random field $T_{L}\(x,t\)$ describes exactly a band-limited random field, with band-width $L_0<L$. Moreover, as mentioned for example in \cite{mcewenwiaux}, it provides a very good approximation of a smooth random field, in the sense that its covariance decays fast as $\ell$ grows to infinity. 
\end{remark}

\section{LASSO estimation on the sphere}\label{sec:LASSO}
Here we present LASSO-type estimators for spherical autoregressive kernels under sparsity assumptions. 
More specifically, merging the techniques based on the stability measure presented in \cite{BM:15} with the properties of $\operatorname{SPHAR}(p)$ processes enables the construction of functional estimators for the kernels $\{k_j: j=1,\dots,p\}$.
\subsection{The estimator construction}\label{lasso_est} As introduced in Section \ref{sec:intro}, the spectral decomposition on the sphere allows to reduce the functional penalized minimization problem to the equivalent $\ell_1$-penalized problems in the space of the harmonic coefficients, see \eqref{eq:philasso} below. Let us recall the definition of our LASSO estimator.

\begin{definition}[LASSO estimator]\label{def:lasso_estimator}
	The functional LASSO estimator for the kernel \eqref{eq:kernel} is defined by
	\begin{equation}\label{eq:lasso}
	\kLASSO  =  \underset{\kvec  \in \PpN }{\argmin} \, \frac1N \sumt \norm {  T\(\cdot, t\) - \sumjp \Phij \( T \(\cdot, t-j\)\) }^2_{\Ltwo} +\lambda \sumjp \TRnorm{ \Phij }\, ,
	\end{equation}
	where $\PpN$ is given by \eqref{eq:PpN}.
\end{definition}

More in details, in line with \cite{cm19}, the first part of \eqref{eq:lasso} can be interpreted as a \textit{functional residual sum of squares}, where each term has spectral decomposition
\begin{align*}
T\(\cdot,t\)-\sum_{j=1}^{p} \(\Phij T\(\cdot,t-j\)\)=  \sumlm \( \alm\(t\) - \sumjp \philj \alm \(t-j\)\)\Y.
\end{align*} 
As a consequence,
\begin{align*}
\sumt \norm{T\(\cdot,t\) - \sumjp \Phi_j \(T\(\cdot, t-j\)\)}^2_{\Ltwo}
&= \sumt \int_{\s} \abs{T\(x,t\) - \sumjp \(\Phi_j T\(\cdot, t-j\)\)(x) }^2 \diff x \\ 
&= \sumt \sumlm \abs{ \alm \(t\) -  \sumjp \philj  \alm\(t-j\)}^2. 
\end{align*}
From now on, we fix a truncation level $L= \LN$, which depends on the number of observations $N$. We can also define the \emph{truncated residual sum of squares} as
\begin{align}\label{eq:sphib}
S\(\phib_0,\dots,\phib_{L-1}\) &= \sumt \sumltrunc \summ \abs{ \alm \(t\) -  \sumjp \philj  \alm\(t-j\)}^2, 
\end{align}
which can be rewritten as the sum of the first $\LN$ components of the functional residual sum of squares, namely,
\begin{equation}\nonumber
S\(\phib_0,\dots,\phib_{\LL}\) =  \sum_{\ell=0}^{\LL} S_\ell (\phib_\ell),
\end{equation}
where
\begin{equation}\nonumber
\Sl  = \sumt \summ \abs{ \alm \(t\) -  \sumjp \philj  \alm \(t-j\) }^2 .
\end{equation}

Using \eqref{eq:sphib}, the functional minimization problem can be reformulated as
\begin{align*}
\kLASSO & = \sumltrunc \phiLASSO
\frac{2\ell+1}{4\pi} \LegP \,,
\end{align*}
with
\begin{align}
\phiLASSO & = \underset{\phib_\ell \in \reals^p} {\argmin}  \, \frac1N \Sl +\lambda(2\ell+1) \norm{\phib_\ell }_1 \notag\\
&=\underset{\phib_\ell \in \reals^p} {\argmin}  \frac{1}{N(2\ell+1)}\Sl+ \lambda\norm{\phib_\ell }_1. \label{eq:philasso}
\end{align}
\begin{remark}[\bf Important remark on the high dimensional nature of the procedure]\label{rmk:highdim}
The formula \eqref{eq:philasso} could lead to a misleading interpretation of the minimization problem. Although from a computational point of view the procedure is \emph{separable}, i.e. it can be solved separately for each $\ell$, the problem still has a high-dimensional nature. First of all, the penalty parameter $\lambda$ does not depend on $\ell$. Moreover, our penalization problem applies also to $\operatorname{SPHAR}(1)$ processes, which means that it can be used as a penalization procedure only on the first $L_N$ multipoles and not on the lags. This fact is made very clear by Figures \ref{fig:1}-\ref{fig:2} where it shows up that, also with very small lag dimension ($p=2$), the LASSO has the effect to regularize the kernel estimates. To better understand our point, take $p=1$ and
$$
\widehat{\phi}^{\operatorname{lasso}}_\ell=\underset{\phi_{\ell} \in \reals} {\argmin}  \frac{1}{N(2\ell+1)} \sum_{t=2}^n \summ \abs{ \alm \(t\) -  \phi_{\ell}  \alm \(t-1\) }^2+ \lambda \abs{\phi_{\ell}}\,;
$$
then the solution can be computed as follows
$$
\widehat{\phi}^{\operatorname{lasso}}_\ell=
\begin{cases}
{\displaystyle \widehat{\phi}^{\operatorname{ols}}_\ell+\frac{\lambda}{2 \widehat{C}_\ell}} & {\displaystyle \widehat{\phi}^{\operatorname{ols}}_\ell\widehat{C}_\ell < - \frac{\lambda}2}\\
 0 & {\displaystyle\abs{\widehat{\phi}^{\operatorname{ols}}_\ell\widehat{C}_\ell} \le \frac{\lambda}2}\\
{\displaystyle\widehat{\phi}^{\operatorname{ols}}_\ell-\frac{\lambda}{2 \widehat{C}_\ell} }& {\displaystyle \widehat{\phi}^{\operatorname{ols}}_\ell\widehat{C}_\ell > \frac{\lambda}2}
\end{cases}\,.
$$
Note that the procedure shrink towards zero not only if the \emph{true} value $\phi_\ell$ is close to zero, but also if the variance $C_\ell$ is very small, which clearly happens at high frequencies, because of the summability of the $C_\ell$'s. On the contrary, for classical LASSO problems, one assumes that the regressors are standardized, which surely cannot happen in our case at fixed $\ell$. Indeed, since here we start our penalization procedure from equation \eqref{eq:lasso},
we can at most rescale by the total variance of the spherical process.
As a consequence, it is now clear that the large $p$, small $N$ ($p \gg N$) scenario, which is an important subject of many of nowadays articles, is not the focus here. 
\end{remark}

\begin{remark}
	Note that the penalization procedure \eqref{eq:lasso} preserves isotropy. This is a consequence of the fact that we are considering a \emph{block-sparsity} model; indeed, given the structure of the predictor $\Phij T(\cdot,t-j)$, 
	where all the $a_{\ell\cdot}$ share the same $\phi_{\ell;j}$, 
	the procedure will automatically select only the relevant multipoles $\ell$.  As a result, a multipole is either removed entirely or not removed at all from the $j$-th component of $\kLASSO$. 
	The reader is referred for further discussions to \cite{CM:15b}, where it is shown that not all the $\ell_1$-penalized problems have solutions which are isotropic, and to \cite{GSWW:18} for sparsity enforcing procedures for isotropic spherical random fields.
\end{remark}

\subsection{Matrix notation}
An alternative form for the minimization problem given by \eqref{eq:lasso} can be introduced as follows. First, we define the following $N\(2\ell+1\)$-dimensional vectors,
\begin{align*}
& \YlN = \( a_{\ell,-	\ell}\(n\),\ldots, a_{\ell, -\ell}\(p+1\),
\ldots, a_{\ell, \ell}\(p+1\)\)', \\
& \YlNh =  \( a_{\ell,-	\ell}\(n-h\),\ldots, a_{\ell, -\ell}\(p+1-h\),
\ldots, a_{\ell, \ell}\(p+1-h\)\)', \quad h=1,\ldots,p,
\\
&\ElN =  \( a_{\ell,-	\ell;Z}\(n\), \ldots, a_{\ell, -\ell; Z}\(p+1\),  \ldots a_{\ell, \ell;Z}(p+1)\),
\end{align*}
where we recall that $\Tup=n-p$. We can thus define the $\(N\(2\ell+1\)\times p\)$ matrix
\begin{equation}\label{eq:XlN}
\XlN = \{ \YlN\(1\):\dots: \YlN\(p\)\}, 
\end{equation}
so that the LASSO problem \eqref{eq:lasso} reduces to
\begin{align*}
\kLASSO & = \sumltrunc \phiLASSO
\frac{2\ell+1}{4\pi} \LegP \,,
\end{align*}
with
\begin{equation}\label{eq:lasso2}
 \phiLASSO =  \underset{\phib_\ell \in \reals^p}{\argmin} \, \frac{1}{N(2\ell+1)}\norm{\YlN - \XlN \phib_\ell}^2_2 +\lambda \norm{\phib_\ell}_1.
\end{equation}
Fixed $\ell = 0, \ldots, \LL$, we define the covariance matrix $\GammalN$, that is, the $p\times p$ matrix with generic  $ij$-th element $C_\ell\(i-j\)$. We can use \eqref{eq:XlN} to define its unbiased estimator 
\begin{equation*}
\GammalNest=\frac{\XlN'\XlN}{N\(2\ell+1\)}. 
\end{equation*}
Let us now consider the product $\XlN'\ElN/N(2\ell+1)$. Observe that $\ElN$ is related to the error random field $Z$, so that we can read this random object as the process obtained from the product of the stochastic data matrix $\XlN$ and the noise vector. Indeed, it represents the so-called \emph{empirical process} (see \cite{bvdg} and the references therein), associated with the multipole $\ell$. Furthermore, observe that the $\ell$-th empirical process corresponds to the following sum
\begin{equation}\label{eq:empproc}
\frac{\XlN'\ElN}{N\(2\ell+1\)} =\frac{\XlN' \YlN}{N\(2\ell+1\)}-\frac{\XlN'\XlN}{N\(2\ell+1\)}\phib_\ell= \gammalN-\widehat\Gamma_{\ell,N}\phib_\ell\,,
\end{equation}
where 
\begin{equation*}
\gammalN = \frac{\XlN' \YlN}{N\(2\ell+1\)}\,.
\end{equation*}
Establishing an upper bound for the empirical processes will be crucial for the proof of the consistency property for the LASSO estimators.

\section{Main results}\label{sec:MR}
In this section, we present the main results of this paper, which consist of properties for the LASSO-type estimator of $\ka$ given in \eqref{eq:lasso}. 
First of all, in Section \ref{sub:stameasure} we introduce the concept of stability measure, a powerful tool firstly proposed in the LASSO framework in \cite{BM:15}, to obtain some bounds on the concentration of the sample covariances and the empirical processes  around their expected values.
Then, in Section \ref{sub:bounds}, we will follow the standard scheme of LASSO-techniques, see \cite{tibshibook1,bvdg}, to establish a \emph{basic inequality}, a \emph{deviation condition} and a \emph{compatibility condition}. Finally, in Section \ref{sub:oracle} we present our main theorem, that is, the so-called \emph{oracle inequalities} for $\kLASSO$.

\subsection{Stability measure on the sphere and deviation bounds} \label{sub:stameasure}
Here, we discuss the stability measure for $\operatorname{SPHAR}(p)$ random fields. Intuitively, a stability measure quantifies the dependencies between the variables of the process and, hence, how stable the autocovariance matrix is. The more intricate the dependencies between the variables are, the less stable should the process result. Several proposals aiming to represent and measure the stability
of a given process have been suggested in the literature over the years, mostly involving set of mixing conditions, in order to establish for how long in time the dependence between the components is effective (see Appendix E in the supplementary file of \cite{BM:15}). The stability measure considered here is in line with the one defined by \cite{BM:15}.\\

Recall that $\{\alm\(t\),  t \in \integers\}$ can be read as a real-valued autoregressive process of order $p$ (see \eqref{eq:arcoeff}). 
Under standard stationarity assumptions (see \cite[page 123]{brockwelldavis}), we can define its spectral density as
\begin{align*}
f_\ell(\nu) =\frac{1}{2\pi} \sum_{\tau = -\infty}^{\infty} \Cl\(\tau\) e^{-i\nu \tau}= \frac{1}{2\pi}\frac{C_{\ell;Z}}{\abs{\phi_\ell(e^{-i\nu})}^2}\,, \qquad \nu \in \[-\pi , \pi\],
\end{align*}
which is bounded and continuous (see also \cite{cm19}). Upper and lower extrema of the spectral density over the unit circle are hence given by 
\begin{align*}
\Mf & =  \max_{\nu \in [-\pi,\pi]} f_\ell(\nu),\\
\mf & =  \min_{\nu \in [-\pi,\pi]} f_\ell(\nu).
\end{align*}
In what follows, we adopt $\Mf$ as a measure of the stability of the process $\{\alm\(t\), t \in \integers\}$. Generalizing \cite{BM:15}, we can consider this as a \textit{band limited} stability measure, in the sense that it refers only to the subprocesses belonging to the multipole $\ell$. A \textit{global} stability measure can be obtained by considering jointly all the multipoles $\ell \geq 0$ via the following definition
\begin{equation*}
\mathcal{M} = \mathcal{M}\(T\) = \max_{\ell\geq 0 } \Mf;
\end{equation*}
whereas we can refer to $\mathcal{M}_N=\max_{\ell<\LN } \Mf$ as the \emph{observed} stability measure. Let us now define the following $p$-dimensional process $$
\tilde{a}_{\ell,m}\(t\) =\(\alm\(t\),\ldots,\alm\(t-p+1\)\)',  
$$
with spectral density and corresponding stability measure given by
$$
\tilde{f}_\ell \(\nu\) = \frac{1}{2\pi}\sum_{\tau=-\infty}^{\infty} \Gamma_\ell \(\tau\)e^{-i \tau \nu} \quad \text{and} \quad M\(\tilde{f}_\ell\)=  \max_{\nu \in [-\pi,\pi]} \Lambda_{\max} (\tilde f_\ell(\nu)), 
$$
where 
$$
\Gamma_\ell \(\tau\) = \Ex\[\tilde{a}_{\ell,m}\(t+\tau\){\tilde{a}}^\prime_{\ell,m}\(t\)\]
$$
and $\GammalN = \GammalN\(0\)$. 
We can therefore construct $\rr$-dimensional subprocesses of $\{\tilde{a}_{\ell,m}\(t\):t\in\integers\}$ as follows. We fix a $\rr$-dimensional index $J=\(j_1,\ldots,j_\rr\)$, so that $J\in\{1,\ldots,p\}^\rr$, and $j_1<\ldots<j_\rr$. Then, we define
$$
\tilde{a}^J_{\ell,m}\(t\) = \(\(\tilde{a}_{\ell,m}\(t\)\)_{j_1},\ldots,\(\tilde{a}_{\ell,m}\(t\)\)_{j_\rr}\)',
$$
where $\(\tilde{a}_{\ell,m}\(t\)\)_{i}$ is the $i$-th component of $\{\tilde{a}_{\ell,m}\(t\):t\in \integers\}$. This subprocess has spectral density $\tilde{f}^J_\ell\(\nu\)$.
We can finally introduce  
\begin{align*}
& 
\mathcal{M}\(\tilde{f}_{\ell},\rr \)= \max_{J \subset \{1, \ldots, p\}, \abs{J}   \le
	\rr}
\mathcal{M}\(\tilde{f}^J_{\ell} \),\\
&\widetilde{\mathcal{M}}(\rr)=\max_{\ell\ge0}\mathcal{M}\(\tilde{f}_{\ell},\rr \),\\
&\widetilde{\mathcal{M}}_N(\rr)=\max_{\ell<\LN}\mathcal{M}\(\tilde{f}_{\ell},\rr \),
\end{align*}
respectively, the band-limited, the global and the observed stability measures of the subprocess $\{\tilde{a}^\rr_{\ell,m}\(t\):t\in \integers \}$.
Notice that $\mathcal{M}\(\tilde{f}_\ell\)=\mathcal{M}\(\tilde{f}_\ell,p\)$, while, for the sake of completeness, we define $\mathcal{M}\(\tilde{f}_\ell,\rr\)=\mathcal{M}(\tilde{f}_\ell)$, for all $\rr>p$. Moreover, it can be shown that
$$
\mathcal{M}\(\tilde{f}_\ell,1\) \le \mathcal{M}\(\tilde{f}_\ell,2\)\le\cdots\le
\mathcal{M}\(\tilde{f}_\ell,p\) = \mathcal{M}\(\tilde{f}_\ell\).
$$
The following quantities are also well-defined
\begin{equation}\label{eq:mumu}
\mumin= \min_{z\in\C:\abs{z}=1}|\phi_\ell(z)|^2, \quad \mumax= \max_{z\in\C:\abs{z}=1}|\phi_\ell(z)|^2 , \quad \text{for} \quad  \ell\ge0\,.
\end{equation}
\begin{remark}\label{rmk:muminmax}
	Note that, in line with Remark \ref{rmk:infinf}, there exists a positive constant $c$ such that                        	
	$$
	\max_{\ell\ge 0} \mu_{\max; \ell}= \max_{\ell\ge 0} \max_{z\in\C:\abs{z}=1} \left | 1- \sum_{j=1}^p \phi_{\ell;j} z^j \right |^2 \le \(1+\sum_{j=1}^p \max_{\ell \ge 0} \abs{\phi_{\ell;j}}\)^2\le c\,,
	$$
	and
$$
	\min_{\ell \ge 0} \mu_{\min; \ell}=\min_{\ell \ge 0} \min_{z\in\C:\abs{z}=1} \left | 1- \sum_{j=1}^p \phi_{\ell;j} z^j \right |^2 = \min_{\ell \ge 0} \min_{z\in\C:\abs{z}=1} \prod_{j=1}^{d_\ell} |1-\xi^{-1}_{\ell; j} z|^2  \ge ( 1-\xi^{-1}_{\ast} )^{2p} >0\,,
$$
recalling that $\xi_{\ast} = \min_{\ell \ge 0} \min_{j = 1,\dots,d_\ell} | \xi_{\ell; j} |$. These bounds represent the moral counterpart of Proposition 2.2 in \cite{BM:15} and show how the global stability $\mathcal{M}$ behaves
for SPHAR models. Indeed, for a $\operatorname{SPHAR}(1)$ process, the operator $\Phi_1$ plays the same role of the matrix $A_1$ in \cite{BM:15} and $\xi^{-1}_{\ast} = \max_{\ell \ge 0}  |\phi_\ell | = \| \Phi_1\|_{\operatorname{op}}$.  As a consequence $\mathcal{M}$ is
bounded as long as the operator norm of $\Phi_1$ (the counterpart of the spectral radius of $A_1$) is
bounded away from 1.
\end{remark}

We apply now the idea of stability measure to establish some relevant deviation bounds on the covariance estimators and the empirical processes, which will be pivotal to analyse our regression problem.
Note that, $\{\alm\(t\):  t \in \integers\}$, $m=-\ell,\dots,\ell$, can be seen as a tool to provide an alternative notation for the empirical process \eqref{eq:empproc}. Indeed the $h$-th component of $\XlN'\ElN/N(2\ell+1)$ is given by
\begin{align}
\frac{\YlNh'\ElN} {N\(2\ell+1\)} &= \frac{1}{N\(2\ell+1\)}  \summ \sumt \alm\(t-h\) \almZ\(t\)\label{eq:sottoprocessi}.
\end{align}

\begin{proposition}[Deviation bounds]\label{prop:DevBound}
	There exists a constant $c>0$ such that for any $\rr$-sparse vectors $u,v \in \R^p$ with $\norm{u},\norm{v}\le1$, $\rr\ge 1$ and any $\eta\ge0$, it holds that
	\begin{align}\label{eq:vvDB}
	&\P\(\abs{v'\(\GammalNest - \GammalN\)v}> 2\pi \MM \( \widetilde f_\ell, \rr \) \eta \) \le 2e^{-c\,N\(2\ell+1\) \min\{ \eta^2, \eta \}},\\
	\label{eq:uvDB}
	&\P\(\abs{u'\(\GammalNest - \GammalN \) v} > 6 \pi \MM \(\widetilde f_\ell,2\rr\) \eta \) \le 6e^{-cN\(2\ell+1\) \min\{ \eta^2, \eta\}}.
	\end{align} 
	In particular, for any $i,j \in \{1,\dots,p\}$, it holds that
	\begin{equation}\label{eq:ijDB}
	\P\(\abs{\(\GammalNest - \GammalN \)_{ij}}>6\pi \mathcal M(\widetilde f_\ell,2)\eta\)\le 6e^{-c\,N(2\ell+1) \min\{ \eta^2, \eta\}}\,.
	\end{equation}
	Moreover, for all $1\le h \le p$, it holds that
	\begin{equation}\label{eq:yeDB}
	\P\(\abs{\frac{\YlN'\(h\) \ElN}{N\(2\ell+1\)}} > 2 \pi\, \ClZ \( 1+\frac{1+\mumax}{\mumin}\)\eta\)\leq 6 e^{-c N\(2\ell+1\) \min \{\eta^2, \eta \}},
	\end{equation}
	where $\mumax$ and $\mumin$ are defined by \eqref{eq:mumu}. 
\end{proposition}

\begin{remark}
Note that the deviation bounds can be also stated in terms of the global or observed stability measure $\widetilde{\mathcal{M}}(\rr)$ or $\widetilde{\mathcal{M}}_N(\rr)$. 
\end{remark}

The analogous result presented in \cite{BM:15} is very general, since it deals with stationary Gaussian random processes on $\reals^d$, while here we focus on deviation bounds for our specific empirical covariance matrices and empirical processes. 
The main technical difference between our results and the ones in \cite{BM:15} is that, in our framework, we use observations from a group of $2\ell+1$ stationary processes, namely, $\{\alm(t): t\in\Z\}$, $m=-\ell, \ldots, \ell$, to estimate the same covariance matrix $\GammalN$, exploiting the isotropy of the random field. 

Similarly to the work \cite{BM:15}, \eqref{eq:vvDB} and \eqref{eq:yeDB} quantify how the underlying estimators concentrate around their expected values. 
In particular, \eqref{eq:vvDB} will be used to verify the compatibility condition (see Proposition \ref{DevCond}), while \eqref{eq:yeDB} will be used to prove the deviation condition (see Proposition \ref{RE_thm}). In the $\text{i.i.d.}$ case, bounds on the empirical process can be easily established, since the data matrix is deterministic and the randomness comes only from the noise vector. In our case, similarly to \cite{BM:15}, the $\ell$-th empirical process is the product of a dependent noise vector and a stochastic data matrix. Therefore, proving consistency requires a bound on both these two random objects.


\subsection{Bounds for LASSO techniques}\label{sub:bounds}
We are now in the position to present the classical path of LASSO in our setting. \\
The very first result concerns the so-called \emph{basic inequality}, an elementary yet essential result, which does not require any condition or assumption, except the existence of a linear underlying model, and it is simply based on the definition of the LASSO estimator. 
\begin{proposition}[Basic inequality]\label{prop:basic} Consider the estimation problem \eqref{eq:lasso}. For any $\ell=0,\dots,\LL$, set $v_\ell= \phiLASSO-\phib_\ell$. 
	Then, the following basic inequality holds
	\begin{equation}\label{eq:BI}
	v_\ell' \GammalNest v_\ell \leq \frac{2 v_\ell'\XlN'\ElN}{N(2\ell+1)} + \lambda \, \big[\norm{\phib_\ell}_1-\norm{\phib_\ell+v_\ell}_1\big]\,.
	\end{equation}
\end{proposition}
This simple result implies that the prediction error $v_\ell' \GammalNest v_\ell$ is bounded by the sum of two factors. The first one is random and it depends on the empirical process $\XlN'\ElN/N\(2\ell+1\)$. The second one is deterministic and its value depends on the penalty parameter $\lambda$, on $\Tup$ and on the chosen linear model itself. \\

The second step consists in defining an event $\S_N$ such that the fluctuations of the random factors 
$$
\frac{2 v_\ell'\XlN'\ElN}{N(2\ell+1)}\,, \quad \ell=0,\dots, \LL\,,
$$ 
when conditioned to $\S_N$, are all controlled by the same deterministic quantity. Moreover, we need to prove that this event has a high probability, implying that a bound on the prediction errors can be obtained in most cases. 
The event $\S_N$ is defined as follows. 
\begin{definition}
	In the setting previously described, let 
	\begin{equation*}
	\S_N=\bigcap_{\ell=0}^{L_N-1} \{\norm{\widehat\gamma_{\ell,N}-\widehat\Gamma_{\ell,N}\phib_\ell}_{\infty} \leq \mathcal F_N\sqrt{\frac{\log pL_N}{N}}\},
	\end{equation*}
	where $\mathcal{F}_N$ is a deterministic function depending only on the parameters $\(\phib_0,\ldots,\phib_{\LL}\)$ and noise variances $\(C_{0;Z},\ldots,C_{\LL;Z}\)$. The deviation condition is said to hold if the event $\S_N$ happens.
\end{definition}
The following theorem shows that, for an appropriate choice of $\mathcal{F}_N$ and $\LN$, the event $\S_N$ has high probability to occur.
\begin{proposition}[Deviation condition]\label{DevCond}
	Consider the regression problem \eqref{eq:arp} and the proposed estimator $\kLASSO$ described in \eqref{eq:lasso}. Assume that Conditions \ref{cond:identify} and \ref{cond:stazionarity} hold. There exist some constants $c_0,c_1,c_2>0$ such that, if we define 
	$$
	\mathcal F\(\phib_\ell,C_{\ell;Z}\)=c_0\[C_{\ell;Z} \(1+\frac{1+\mu_{\max; \ell}}{\mu_{\min; \ell}} \)\], \qquad \mathcal F_N= \max_{\ell< L_N} \mathcal F\(\phib_\ell,C_{\ell;Z}\), 
	$$
	and if $N \succeq \log p L_N$, then
	$$
	\P\(\bigcap_{\ell=0}^{L_N-1} \norm{\widehat\gamma_{\ell,N}-\widehat\Gamma_{\ell,N}\phib_\ell}_{\infty} \leq \mathcal F_N \sqrt{\frac{\log p L_N}{N}}\)\geq 1-c_1e^{-c_2\log \(p\LN\)}\,.
	$$
\end{proposition}
\begin{remark}Observe that:
	\begin{itemize} 
		\item[(i)] the results presented in Proposition \ref{DevCond} also hold for a different choice of $\mathcal{F}_N$, that is, 
		$$
		\mathcal F_N= c_0\[\(\max_{\ell< L_N}C_{\ell;Z}\) \(1+\frac{1+\max_{\ell< L_N} \mu_{\max ; \ell }}{\min_{\ell< L_N}\mu_{\min; \ell}} \)\],
		$$
		which corresponds to the one used in \cite{BM:15};
		\item[(ii)] in order for this bound to make sense, we need that $$\log\(p\LN\)=o\(N\).$$ 
	\end{itemize}
\end{remark}

The third and final step is to establish a \emph{compatibility condition} that, whenever verified on the event $\S_N$, will allow us to bound both the prediction errors $\{\norm{\XlN \(\phiLASSO - \phib_\ell\) }_2^2\}$ and the estimation errors $\{\norm{\phiLASSO - \phib_\ell}_2^2\}$ by the same quantity. In this sense, it makes the errors compatible.

A symmetric $d \times d$ matrix $A$ satisfies the compatibility condition, also called restricted eigenvalue (RE) condition, with curvature $\alpha>0$ and tolerance $\tau>0$ ($A\sim RE(\alpha,\tau)$), if, for any $\theta \in \R^d$, 
\begin{equation}\label{RE}
\theta'A\theta\ge \alpha\norm{\theta}_2^2-\tau\norm{\theta}_1^2\,.
\end{equation}  

The next result gives some sufficient conditions in order to have $$\widehat\Gamma_{\ell,N} \sim RE(\alpha,\tau),$$ for some $\alpha$ and $\tau$, with high probability.

\begin{proposition}[Compatibility condition]\label{RE_thm}
	Consider the estimation problem \eqref{eq:lasso} and assume that Conditions \ref{cond:identify} and \ref{cond:stazionarity} hold. Define $q_N=\max_{\ell < \LN}q_\ell$. There exist some constants $c_1,c_2,c_3>0$ such that, if
	\begin{equation}\label{eq:RE-cond}
	N \succeq \max\{\omega_N^2,1\} q_N \log \(p \LN\), \text{ with } \quad \omega_N=c_3\max_{\ell<\LN} \frac{\mumax}{\mumin} ,
	\end{equation}
	then
	\begin{equation}\label{eq:RE_thm}
	\P\(\bigcap_{\ell=0}^{L_N-1}\{ \widehat\Gamma_{\ell,N}\sim RE(\alpha_\ell,\tau_\ell)\}\)\ge 1-c_1\,e^{-c_2\,N\,\min\{\omega_N^{-2},1\}},
	\end{equation}
	with
	\begin{equation}\label{anna}
	\alpha_\ell= \frac{\ClZ}{2\,\mumax}, \quad \text{and} \quad \tau_\ell=\alpha_\ell \, \max\{\omega_N^2,1\}\frac{\log \(p \LN\)}{N} . 
	\end{equation}
\end{proposition}
\begin{remark}
	The results presented in Proposition \ref{RE_thm} also hold for  $$\omega_N=c_3\frac{\max_{\ell< \LN} \mumax}{\min_{\ell< \LN} \mumin},$$ analogously to the findings in \cite{BM:15}.	
\end{remark}

\begin{remark}
	The compatibility condition on $\widehat\Gamma_{\ell,N}$ is a requirement on its smallest eigenvalue, which can be seen as a measure of the dependence of the random matrix columns. The sufficient condition \eqref{eq:RE-cond} ensures that, with high probability, the (sample) minimum eigenvalue of the matrix $\widehat\Gamma_{\ell,N}$ is bounded away from zero. 
\end{remark}

\begin{remark}
	Note that $0\le\mathcal{F}_N \le  \mathcal F$, where $ \mathcal F =\max_{\ell\ge0}\mathcal F\(\phib_\ell,C_{\ell;Z}\)$ exists finite. Indeed, 
	\begin{flalign*}
	\mathcal F\(\phib_\ell,C_{\ell;Z}\)&=c_0\[C_{\ell;Z} \(1+\frac{1+\mu_{\max; \ell}}{\mu_{\min; \ell}} \)\]\to 0 , \qquad \text{as } \ell \to \infty,
	\end{flalign*}
	since $C_{\ell;Z}$ converges to zero as $\ell$ goes to infinity and $a\le \mu_{\min; \ell} \le \mu_{\max; \ell} \le b$, with $a,b$ positive constants (independent of $\ell$), see Remark \ref{rmk:muminmax}. Similarly, it holds that $0\le\omega_N \le  \omega$, where 
	$$
	\omega=c_3\max_{\ell\ge0}  \frac{\mumax}{\mumin},$$
	and $q_N\le q=\max_{\ell\ge0}q_\ell$. In particular, all the results presented in this paper can be stated using $\mathcal{F}, \omega, q$ instead of $\mathcal F_N, \omega_N, q_N$. Without loss of generality, we can assume $q_N \ge 1$.
\end{remark}



\subsection{Oracle inequalities}\label{sub:oracle}
Oracle inequalities are used to estimate the accuracy of the $\kLASSO$. Observe that, in general, $\kLASSO$ depends on the penalty parameter $\lambda$, according to \eqref{eq:lasso}. As a consequence, given a proper choice of $\lambda$, oracle inequalities produce upper bounds for the estimation error with high probability. Such upper bounds are characterized by a multiplicative factor $\log (p\LN)$; roughly speaking, this factor is the cost for not knowing explicitly the set of non-zero coefficients.
\begin{theorem}\label{th:MR-preciso}
	Consider the estimation problem \eqref{eq:lasso} and assume that Conditions \ref{cond:identify} and \ref{cond:stazionarity} hold. Moreover, suppose that, for any $\ell = 0,\ldots,\LL$, $\GammalNest \sim RE(\alpha_\ell,\tau_\ell)$ with $\ql\tau_\ell \le \alpha_\ell/32$ and that 
	$(\widehat\Gamma_{\ell,N}, \widehat\gamma_{\ell,N})$ satisfies the deviation condition almost surely, that is,
	\begin{equation}\label{DB}
	\norm{\widehat\gamma_{\ell,N}-\widehat\Gamma_{\ell,N}\phib_\ell}_{\infty}\leq \mathcal F_N\, \sqrt{\frac{\log\( p \LN\)}{N}} \quad a.s. \quad .
	\end{equation}
	Then, for any $\lambda = \lambda_{N} \ge4 \mathcal F_N\sqrt{\log (pL_N)/N}$, where $ \mathcal F_N =\max_{\ell<L_N}\mathcal F\(\phib_\ell,C_{\ell;Z}\)$, any solution $\widehat{\mathbf{k}}^{\textnormal{lasso}}_N$ of \eqref{eq:lasso} satisfies 
	\begin{flalign}
	\norm{\widehat{\mathbf{k}}^{\textnormal{lasso}}_N-\mathbf{k}}^2_{L^2} &\le \frac{18}{\pi^2}\, \lambda_{N}^2\,\sum_{\ell=0}^{\LL}\,\frac{\ql}{\alpha_\ell^2} \,(2\ell+1)+\Big \| \ka - \ka_{N} \Big  \|^2_{L^2}\,;\label{eq:oraclek1}
	\end{flalign}
	moreover, under the additional Condition \ref{cond:smooth}, it holds that
	\begin{flalign}
	\norm{\widehat{\mathbf{k}}^{\textnormal{lasso}}_N-\mathbf{k}}_{L^\infty} &\le \frac{3}{\pi}\, \lambda_{N} \,\sum_{\ell=0}^{\LL}\,\frac{\sqrt{\ql}}{\alpha_\ell} \,(2\ell+1)+\Big \| \ka - \ka_{N} \Big \|_{L^\infty}\,.\label{eq:oraclek2}
	\end{flalign}
\end{theorem}


\paragraph{Rates of convergence} 

To discuss the possible rates of convergence in \eqref{eq:oraclek1}-\eqref{eq:oraclek2}, let us choose $\lambda_N=4 \mathcal F_N\sqrt{\log (pL_N)/N}$ and $L_N\sim N^d$. Moreover, we impose some semiparametric structure to the set $\{\phi_{\ell;j}:\ell\geq 0\}$, that is,  
\begin{equation*}
\abs{\phi _{\ell;j}}\leq G_j\ell ^{-\beta _{j}},  
\end{equation*}
where $\beta_{j}>1$ and $G_{j}>0$ (see also \cite{cm19}). 
Note that, since we are looking at the asymptotic behaviour as $N\sto\infty$, the sufficient condition \eqref{eq:RE-cond} automatically holds and the curvature $\alpha_\ell\sim \ClZ$. In particular, a standard assumption for the behaviour of the power spectrum of a spherical white noise is $\ClZ\sim \ell^{-\alpha}$, with $\alpha>2$, see \cite{MP:11}.
As a consequence, in this framework, we have	
\begin{flalign*}
\norm{\widehat{\mathbf{k}}^{\textnormal{lasso}}_N-\mathbf{k}}^2_{L^2} &\le \frac{18}{\pi^2}\, \lambda_{N}^2\,\sum_{\ell=0}^{\LL}\,\frac{\ql}{\alpha_\ell^2} \,(2\ell+1)+\sum_{\ell=\LN}^\infty \norm{\phib_\ell}_2^2\frac{2\ell+1}{8\pi^2} \\
&\le const \,\[ \frac{\log N}{N} \,\sum_{\ell=0}^{\LL}\,\ell^{2\alpha}\,(2\ell+1)+\sum_{\ell=\LN}^\infty \ell^{-2\beta} (2\ell+1)\]\\
&=O\( \log N \, N^{2d(\alpha+1)-1}+ N^{2d(1-\beta)}\),
\end{flalign*}
where $\beta=\min_{j=1,\dots,p}\beta_j$, and, in order to ensure consistency, we can choose 
$$0<d < \frac{1}{2(\alpha + 1)}\,.$$
Analogously, imposing this time $\beta_j>2$, one has that
\begin{flalign*}
\norm{\widehat{\mathbf{k}}^{\textnormal{lasso}}_N-\mathbf{k}}_{L^\infty} &\le \frac{3}{\pi}\, \lambda_{N} \,\sum_{\ell=0}^{\LL}\,\frac{\sqrt{\ql}}{\alpha_\ell} \,(2\ell+1)+\sum_{\ell=\LN}^\infty \norm{\phib_\ell}_2\frac{2\ell+1}{4\pi}\\
&=O\( \(\log N\)^{1/2} \, N^{d(\alpha+2)-\frac{1}{2}}+ N^{d(2-\beta)}\)\,
\end{flalign*}
and in this case the consistency in the supremum norm is reached for any 
$$0<d<\frac{1}{2(\alpha + 2)}\,.$$
Note that, just like in the multivariate (i.e. non-functional) case, not knowing the sparsity pattern only comes
at a logarithmic cost in performance (see Theorem 18 in \cite{cm19}).
Moreover, we stress that the parameter $\alpha$ can be estimated via a Whittle-like procedure, see \cite{dlmejs,dlm}, which could be useful to choose a suitable rate for $L_N$.


\section{Some numerical results}\label{sec:simulations}
In this Section we will describe the numerical implementation to support the results provided along the paper. More specifically, here we briefly discuss the performance of the LASSO estimator $\kLASSO$ under sparsity assumptions. 

Fixed $N=300$ and $L_N = L = 50$, we are concerned with the empirical evaluation of the $L^2$-risk of $\kLASSO$ for the penalty parameters $\lambda_{i} = 2\cdot10^{i-6}$, $i=1,\ldots,6$, in comparison with the one of the non-penalized estimator described in \cite{cm19}, corresponding to the case $\lambda_0=0$. 
We remark that our simulations can be considered as a hint, in view of future applications on real data.
In what follows, we consider four different case studies, all belonging to the class of $\operatorname{SPHAR}(2)$ processes, so that \eqref{eq:sphar} becomes $$T\(x,t\)= \Phi_1 \(T\(\cdot,t-1\)\) \(x\)+\Phi_2 \(T\(\cdot,t-2\)\) \(x\)+Z\(x,t\).$$ 

In the first case, the random field $T_1$ is strongly sparse, in the sense that the only non-null eigenvalues are $$\phi_{2;1}=-0.7, \quad \phi_{3;2}=0.5.$$ In the second case, the random field $T_2$ is characterized by less sparsity; in particular, the non-null coefficients are $$\phi_{30;1}=-0.72,\quad \phi_{31;1}=0.31,\quad \quad\phi_{32;1}=0.85, \quad \phi_{2,2}=0.25, \phi_{3,2}=-0.87,\quad\phi_{5,2}=-0.98.$$ 
In the third case, the random field $T_3$ is not sparse, even if the eigenvalues are taken to be relevant only on the first $20$ multipoles, that is, $\phi_{\ell,j}\propto\ell^{-2}$ for $\ell\ge20$ and $j=1,2$.
Finally, in the fourth case $T_4$, all the multipoles are assumed to be relevant. 

\begin{table}[H]
	\centering
	\begin{tabular}{|c|c|c|c|c|}
		\hline
		MSE & $T_1$ & $T_2$ & $T_3$ & $T_4$\\ \hline
		$\lambda_0$ & 0.00218 & 0.00212 & 0.00189 & 0.00143\\
		$\lambda_1$ & 0.00214 & 0.00208 & 0.00185 & 0.00144\\
		$\lambda_2$ & 0.00181 & 0.00175 & 0.00153 & 0.00149 \\
		$\lambda_3$ & 0.00083 & 0.00087 & 0.00071 & 0.00650\\
		$\lambda_4$ & 0.00032 & 0.00478 & 0.00190 & 0.46580\\
		$\lambda_5$ & 0.00015 & 0.15363 & 0.10620 & 7.32961\\ 
		$\lambda_6$ & 0.00049 & 0.53786 & 1.57303 & 15.18901\\
		\hline 
	\end{tabular}
	\caption{Values of the mean squared error (MSE) for the four case studies $T_1, \dots, T_4$, by varying the penalty parameter $\lambda$. Note that $\lambda_0$ corresponds to the non-penalized estimation.}
	\label{tab:1}
\end{table}

Table \ref{tab:1} collects the values of the empirical mean squared error (MSE) associated with the four models of interest. In particular, for $B=1000$ replications, we have
\begin{equation*}
\operatorname{MSE}\(\kLASSO,\ka\)=\sum_{j=1}^2 \{\frac{1}{B\cdot G}\sum_{b=1}^B\sum_{g=1}^G\(\widehat{k}^{\operatorname{lasso}}_j(z_g)-k_j(z_g)\)^2\}\,,
\end{equation*}
where $\{z_1,\dots, z_G\}$, $G=2000$, is an equally spaced grid over $[-1,1]$.
As expected, LASSO-type estimators provide smaller MSE when we consider highly sparse models ($T_1$ and $T_2$). Regarding $T_3$, the situation does not really change because, after the $20$-th multipole, the values of the eigenvalues are very small. On the contrary, for the model $T_4$, the penalized estimator always performs poorly when $\lambda$ grows. Note that for $\lambda=\lambda_1$ the penalization is very small and, as a consequence, the penalized and non-penalized estimators are almost equivalent.  

\begin{minipage}{\linewidth}
    \centering
      \begin{minipage}{0.45\linewidth}
          \begin{figure}[H]
          	 \includegraphics[width=\linewidth]{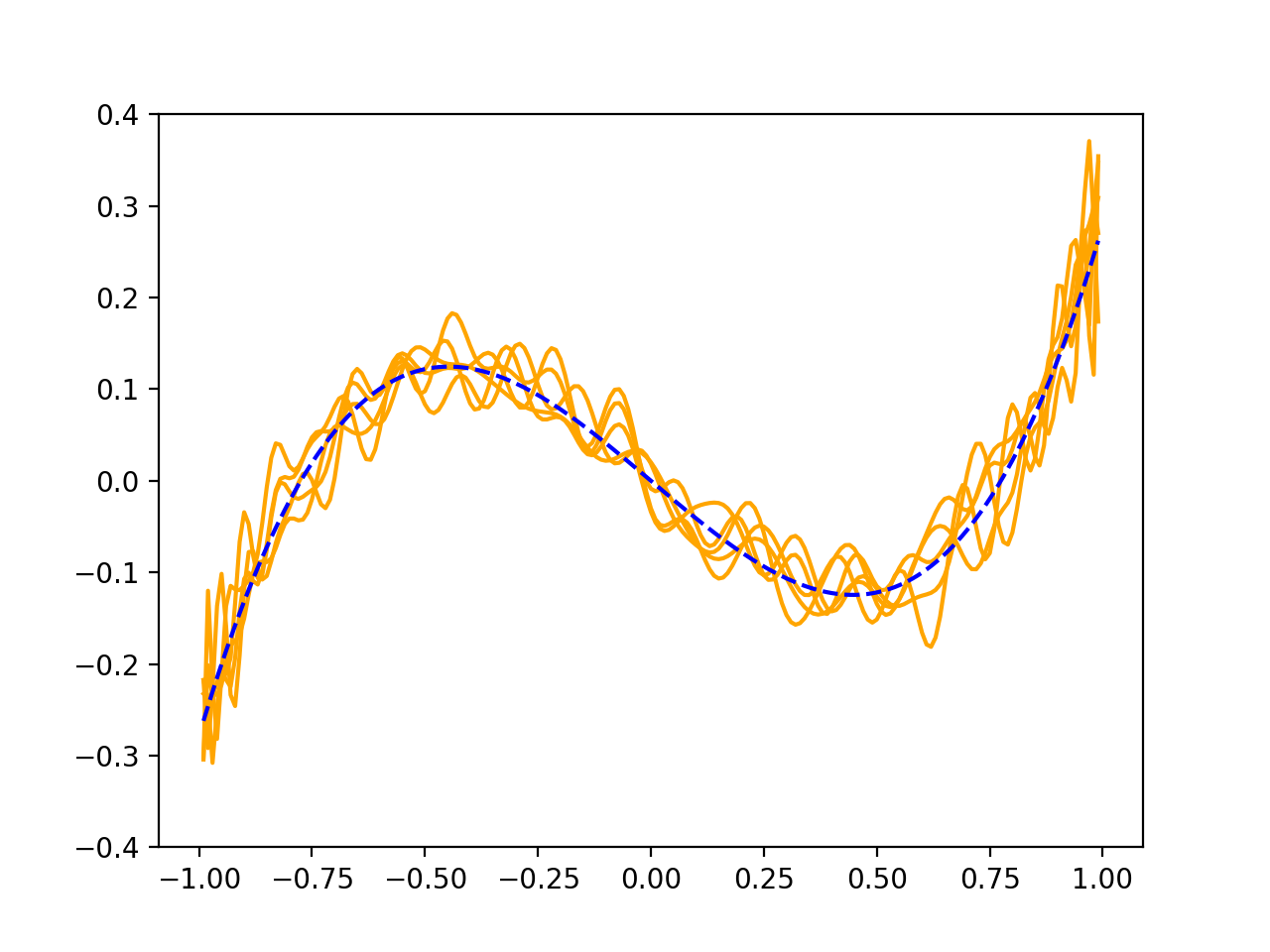} \\
                      \end{figure}
      \end{minipage}
      \hspace{0.05\linewidth}
      \begin{minipage}{0.45\linewidth}
          \begin{figure}[H]
          	 \includegraphics[width=\linewidth]{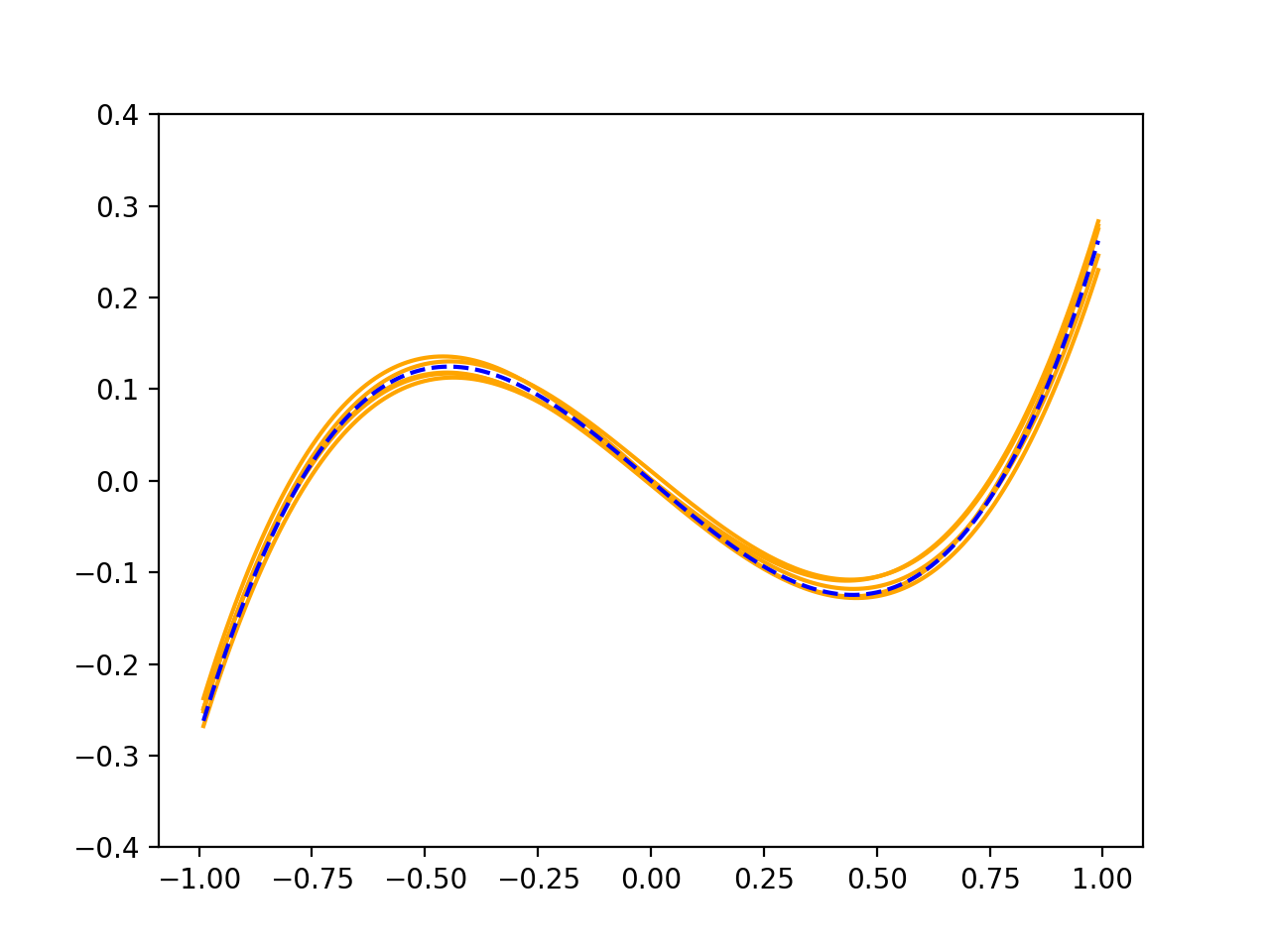} \\
          \end{figure}
      \end{minipage}
      \captionof{figure}{Five estimates (orange lines) of $k_2$ for the model $T_1$, with $\lambda=\lambda_5$. The true $k_2$ is the blue dashed line.}\label{fig:1}
  \end{minipage}
\begin{minipage}{\linewidth}
    \centering
      \begin{minipage}{0.45\linewidth}
          \begin{figure}[H]
               \includegraphics[width=\linewidth]{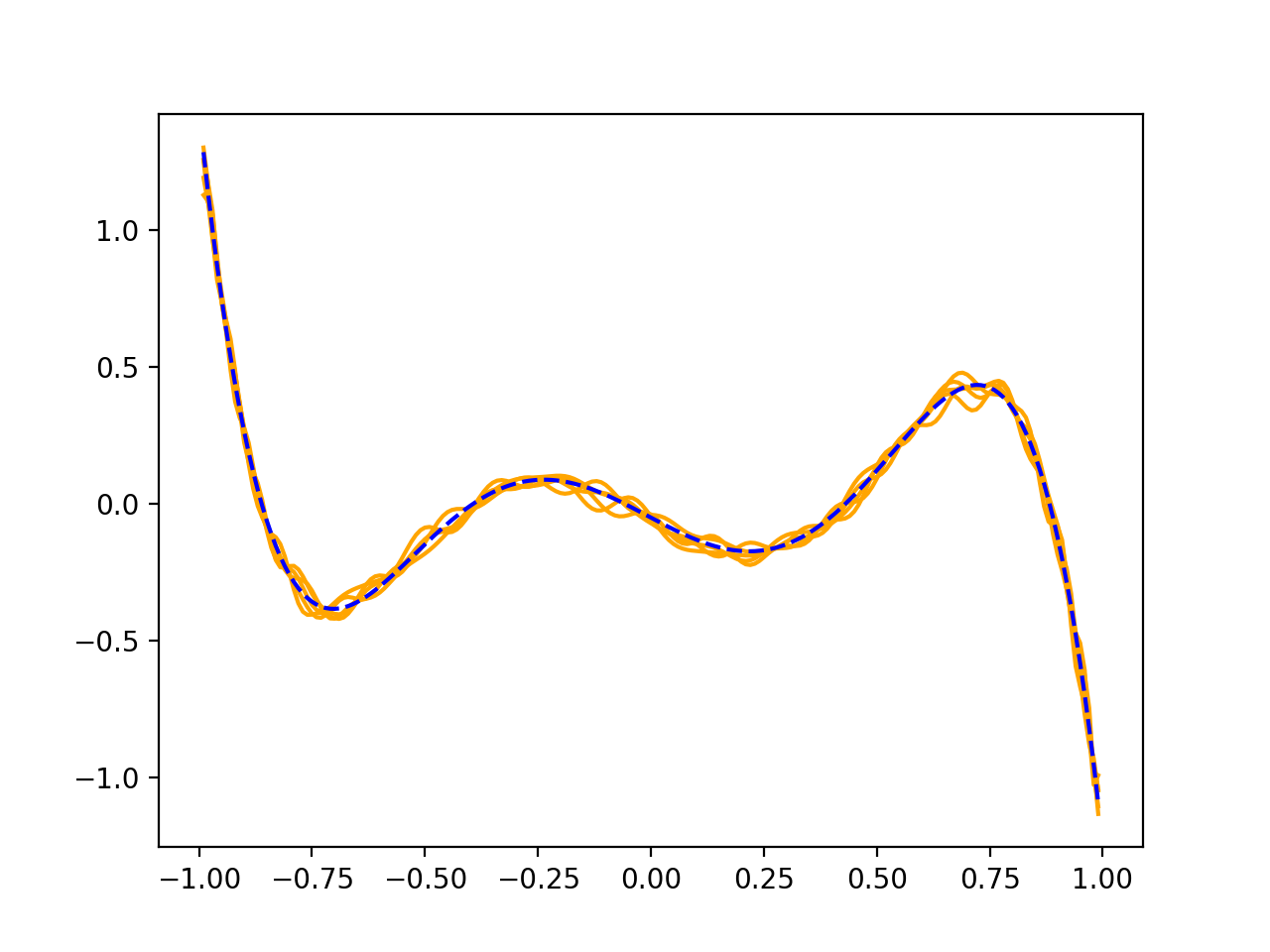} \\
                      \end{figure}
      \end{minipage}
      \hspace{0.05\linewidth}
      \begin{minipage}{0.45\linewidth}
          \begin{figure}[H]
               \includegraphics[width=\linewidth]{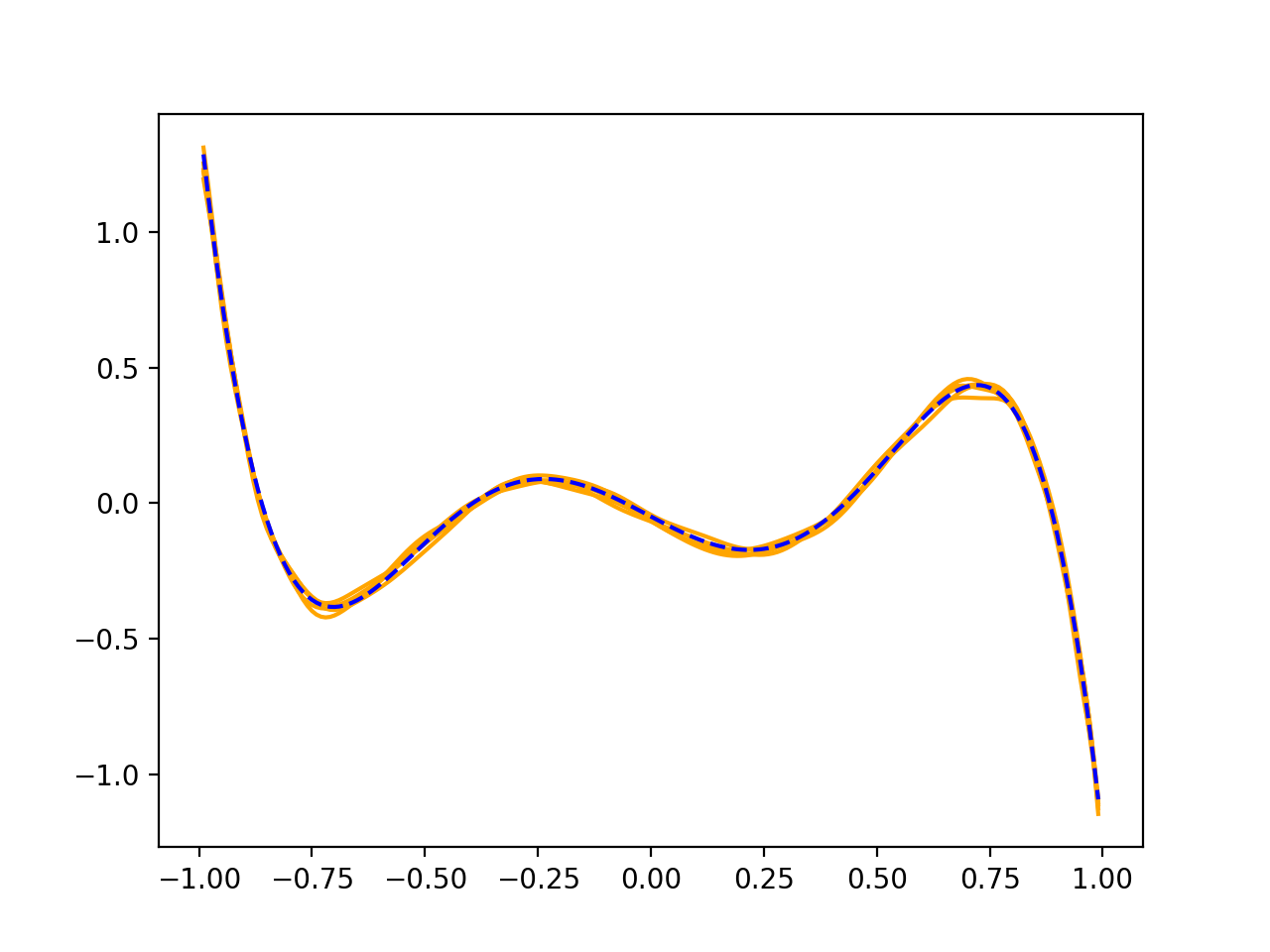} \\
          \end{figure}
      \end{minipage}
      \captionof{figure}{Five estimates (orange lines) of $k_2$ for the model $T_2$, with $\lambda=\lambda_3$. The true $k_2$ is the blue dashed line.}\label{fig:2}
  \end{minipage}

\begin{minipage}{\linewidth}
    \centering
      \begin{minipage}{0.45\linewidth}
          \begin{figure}[H]
                              \includegraphics[width=\linewidth]{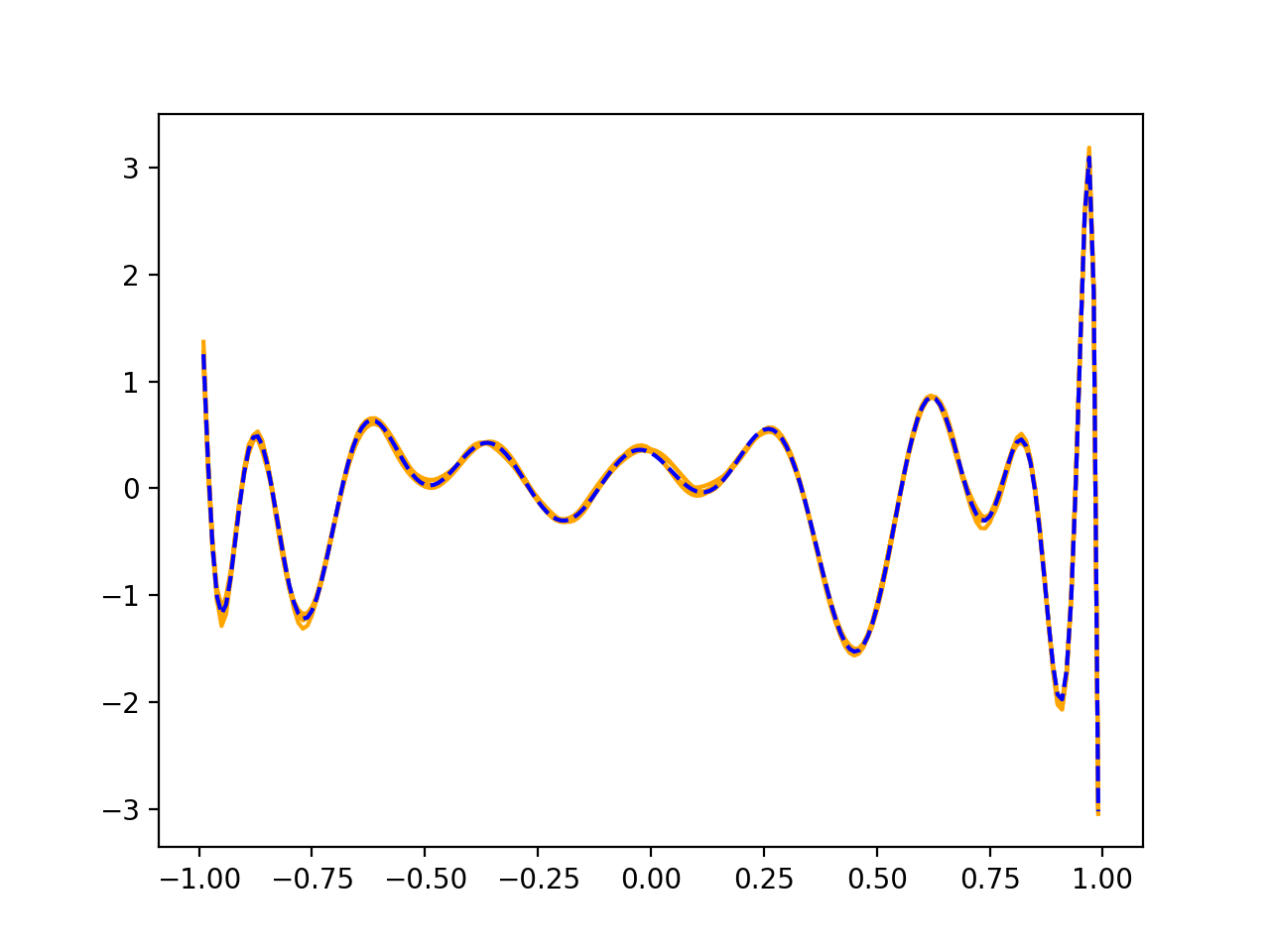} \\
          \end{figure}
      \end{minipage}
      \hspace{0.05\linewidth}
      \begin{minipage}{0.45\linewidth}
          \begin{figure}[H]
                              \includegraphics[width=\linewidth]{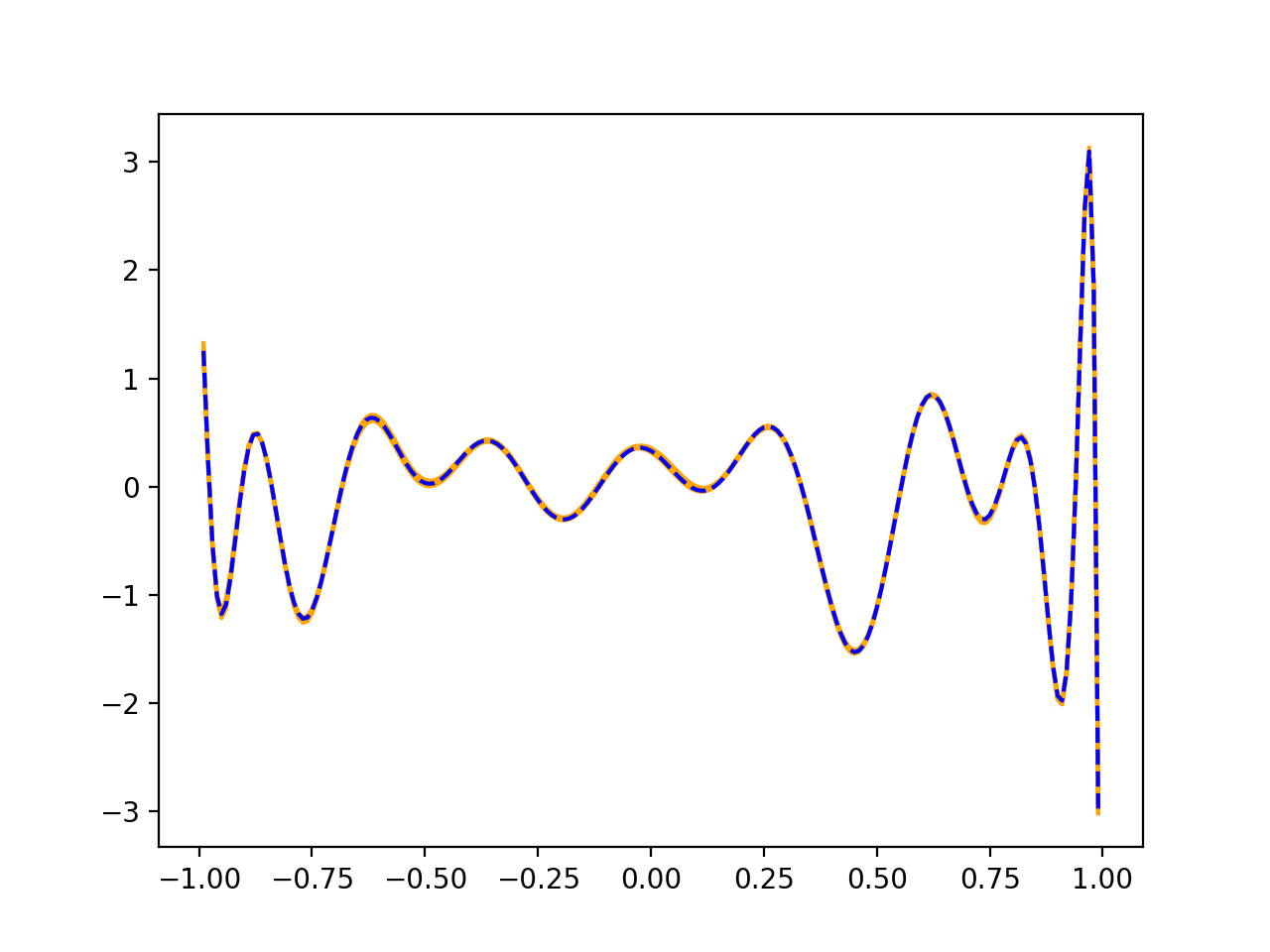} \\
          \end{figure}
      \end{minipage}
      \captionof{figure}{Five estimates (orange lines) of $k_2$ for the model $T_3$, with $\lambda=\lambda_3$. The true $k_2$ is the blue dashed line.}\label{fig:3}
  \end{minipage}

\bigskip

Figures \ref{fig:1}-\ref{fig:3} illustrate the functional forms of the estimated kernel, compared to the real one. The left panels present the non-penalized estimates of $k_2$ while the right panels contain the corresponding penalized ones. Heuristically, the (best) penalized estimates for $T_1$ and $T_2$ reconstruct the true kernel function better than the non-penalized ones, which show an oscillatory behavior (undersmoothing) due to the lack of selection of the relevant multipoles. 
For the model $T_3$, which is non-sparse but has few relevant multipoles, the difference between the two functional estimations is not significant. 

\begin{remark}[{\bf Penalty parameter selection}]
Looking at Table \ref{tab:1} it is clear that the role of the penalty parameter is crucial. 
Our theoretical results suggest possible rates of the penalty parameter that imply consistency of the kernel estimator. However, we cannot directly compute the optimal value of $\lambda$ due to the presence of unknown constants. As a consequence, a careful selection of the penalty parameter using data driven methods needs to be conducted in order to make our procedure applicable in practice. The idea could be to implement a \emph{functional cross-validation}, which takes into account all the frequencies simultaneously; this is not the focus of our work but it can be addressed in future more applied works.  
\end{remark}

\section{Proofs}\label{sec:proofs}
In the present section we prove the bounds showed in Section \ref{sub:bounds} as well as our main theorem. Many of the proofs have arguments which are broadly similar to those given for related results in \cite{BM:15}.

\begin{proof}[Proof of Proposition \ref{prop:DevBound}]
	Let us define $J=\text{supp}(v)=\{j_1,\dots,j_\rr\}\subset\{1,\dots,p\}$, $\rr \ge 1$, and 
	$$
	W_{\ell,J}=X_{\ell,N}v=\sum_{j\in J}v_j \mathbf{Y}_{\ell, N}(j)\,.
	$$
	Then, $Q_{\ell,J}=\Ex\[W_{\ell,J}W_{\ell,J}'\]=B_{\ell,J}\otimes I_{2\ell+1}$, where $B_{\ell,J}$ is the covariance matrix of the random vector 
	$$
	\sum_{j\in J}v_j\begin{bmatrix}
	a_{\ell,m}(n-j)  \\
	\vdots \\
	a_{\ell, m}(p+1-j)
	\end{bmatrix} 
	\qquad \text{for any $m=-\ell,\dots,\ell$. }
	$$
	As a consequence, $\norm{Q_{\ell,J}}_{op}=\norm{B_{\ell,J}}_{op}\le 2\pi \mathcal M(\widetilde f_\ell,\rr)$ (see \cite[Proposition 2.4]{BM:15}) and \eqref{eq:vvDB} is proved. 
	To prove \eqref{eq:uvDB}, note that
	\begin{align*}
	&2 \abs{ u' \(\widehat\Gamma_{\ell,N}-\Gamma_\ell \) v
	}  \le \abs{ u' \(\widehat\Gamma_{\ell,N}-\Gamma_\ell
		\)u } + \abs{ v' \(\widehat\Gamma_{\ell,N}-\Gamma_\ell\)v }+ \abs{ (u+v)' \(\widehat\Gamma_{\ell,N}-\Gamma_\ell
		\) (u+v) },
	\end{align*}
	and $u+v$ is $2r$-sparse with $\abs{  u +v}   \le2$. The result follows by
	applying \eqref{eq:vvDB} separately on each of the three terms
	on the right.
	The element-wise deviation bound \eqref{eq:ijDB} is obtained by choosing $u=e_i$, $v=e_j$.
	
	Let us now prove \eqref{eq:yeDB}. Recall \eqref{eq:sottoprocessi}; the following decomposition holds
	\begin{align*}
	&\frac{2}{N(2\ell+1)}\sum_{m=-\ell}^\ell \sum_{t=p+1}^n a_{\ell, m}(t-h)a_{\ell, m;Z}(t) \\	& \qqqquad= \[ \frac{1}{N(2\ell+1)} \sum_{m=-\ell}^\ell \sum
	_{t=p+1}^n \( \alm(t-h) +\almZ (t)\)^2- (C_\ell + C_{\ell;Z}) \]\\
	& \qqqquad- \Biggl[ \frac{1}{N(2\ell+1)} \sum_{m=-\ell}^\ell \sum_{t=p+1}^n
	\alm(t-h)^2 -C_\ell \Biggr]  - \Biggl[ \frac{1}{N(2\ell+1)} \sum_{m=-\ell}^\ell\sum
	_{t=p}^T \almZ(t)^2-C_{\ell;Z} \Biggr] .
	\end{align*}
	Implementing \eqref{eq:vvDB} for $v=e_h$, we have 
	$$
	\P \(\abs{  \frac{1}{N(2\ell+1)} \sum_{m=-\ell}^\ell \sum_{t=p+1}^n
		\alm(t-h)^2 -C_\ell} > 2\pi\mathcal{M}(\widetilde{f}_\ell,1) \eta \) \le 2e^{-c\,N(2\ell+1) \min\{ \eta^2, \eta\}}\,, 
	$$
	which implies
	\begin{equation*}
	\P \(\abs{ \frac{1}{N(2\ell+1)} \sum_{m=-\ell}^\ell \sum_{t=p+1}^n
		\alm(t-h)^2 -C_\ell} > 2\pi\mathcal{M}(f_\ell) \eta \) \le 2e^{-c\,N(2\ell+1) \min\{ \eta^2, \eta\}}\,,
	\end{equation*}
	where we used the fact that $\mathcal{M}(\widetilde{f}_\ell,1)=\mathcal{M}(f_\ell)$.
	Following steps that are analogous to the ones that led to \eqref{eq:vvDB} (setting $v\in \R$, $v=1$, and obviously $\rr=1$), one can show that 
	$$
	\P \(\abs{ \frac{1}{N(2\ell+1)} \sum_{m=-\ell}^\ell\sum_{t=p}^T \almZ(t)^2-C_{\ell;Z}} > 2\pi\mathcal{M}(f_{	\ell;Z}) \eta \) \le 2e^{-c\,N(2\ell+1) \min\{ \eta^2, \eta\}}\,, 
	$$
	and that, for any fixed $h=1,\dots, p$, 
	\begin{equation*}
	\begin{split}
	&\P \(\abs{ \frac{1}{N(2\ell+1)} \sum_{m=-\ell}^\ell \sum
		_{t=p+1}^n \( \alm(t-h) +\almZ (t)\)^2- (C_\ell + C_{\ell;Z})} > 2\pi\mathcal{M}(f_{\ell; T+Z}) \eta \) \\
	&\qqqquad\qqqquad\qqqquad\le 2e^{-c\,N(2\ell+1) \,\min\{ \eta^2, \eta\}}\,, 
	\end{split}
	\end{equation*}
	Moreover,
	$$
	f_{\ell;T+Z}(\lambda)=f_{\ell}(\lambda)+f_{\ell; Z}(\lambda)+2f_{\ell;(T,Z)}(\lambda)\,,
	$$
	which implies, $\mathcal{M}(f_{\ell; T+Z}) \le  \mathcal{M}(f_{\ell})+\mathcal{M}(f_{\ell;Z})+\mathcal{M}(f_{\ell;(T,Z)})$, where $
	f_{\ell;T+Z}(\lambda)$ is the spectral density of the process $\{\alm(t-h)+\almZ(t), \,t\in\Z\}$ and $f_{\ell;(T,Z)}$ is the spectral density of the joint process $\{\(\alm(t-h),\almZ(t)\), \,t\in\Z\}$.
	
	Now, using the obvious implications
	$$
	\{\abs{X_1+X_2+X_3}>a\} \subset \{\abs{X_1}+\abs{X_2}+\abs{X_3}>a\} \subset \bigcup\limits_{i=1}^3\{\abs{X_i}>\frac{a}{3}\}\,,
	$$  
	we have
	\begin{align*}
	&\P \(\abs{\frac{2}{N(2\ell+1)}\sum_{m=-\ell}^\ell \sum_{t=p+1}^n \alm(t-h)\almZ(t)} > 2\pi \( \mathcal{M}(f_{\ell})+\mathcal{M}(f_{\ell;Z})+\mathcal{M}(f_{\ell;(T,Z)})\)\eta \) \\
	&\qqqquad\qqqquad\qqqquad	\le 6 e^{-c\,N(2\ell+1) \min\{ \eta^2, \eta\}}\,, 
	\end{align*}
	Following the last steps of the proof of Proposition 2.4 in \cite{BM:15}, we obtain
	$$
	2\pi\mathcal{M}(f_{\ell})\le \frac{C_{\ell;Z}}{\mu_{\min; \ell}}\,, \quad 2\pi\mathcal{M}(f_{\ell;Z})= C_{\ell;Z}\, \quad \text{and}\quad 2\pi\mathcal{M}\(f_{\ell;(T,Z)}\)\le \frac{C_{\ell;Z}\,\mu_{\max; \ell}}{\mu_{\min; \ell}}\,,
	$$
	which finally implies \eqref{eq:yeDB}.

\end{proof}

\begin{proof}[Proof of Proposition \ref{prop:basic}]
	Since $\phiLASSO$ is the solution of the minimization problem \eqref{eq:lasso2}, it follows that
	$$
	\frac{\norm{\mathbf{Y}_{\ell,N}-X_{\ell,N}\phiLASSO}^2_2}{N(2\ell+1)}+\lambda \norm{\phiLASSO}_1\leq \frac{\norm{\YlN - \XlN\phib_\ell}^2_2}{N(2\ell+1)}+\lambda \norm{\phib_\ell}_1\,,
	$$
	which, using the definition of $\mathbf{Y}_{\ell,N}$, becomes
	$$
	\frac{\norm{X_{\ell,N} \phib_\ell+ \ElN -\XlN \phiLASSO}^2_2}{N(2\ell+1)}+\lambda \norm{\phiLASSO}_1\leq \frac{\norm{\ElN}^2_2}{N(2\ell+1)}+\lambda \norm{\phib_\ell}_1\,.
	$$
	Now, we have
	$$
	\frac{\norm{X_{\ell,N} \phib_\ell-X_{\ell,N}\phiLASSO}^2_2}{N(2\ell+1)}+\frac{\norm{\ElN}^2_2}{N(2\ell+1)}-2\frac{\(\phiLASSO-\phib_\ell\)'X_{\ell,N}'\ElN}{N(2\ell+1)}+ \lambda  \norm{\phiLASSO}_1
	$$
	$$
	\leq \frac{\norm{\ElN}^2_2}{N(2\ell+1)}+ \lambda  \norm{\phib_\ell}_1,
	$$
	and finally, using the notation $v_\ell$, it holds that
	$$
	\frac{v_\ell 'X_{\ell,N}'X_{\ell,N}v_\ell}{N(2\ell+1)}-2\,\frac{v_\ell'X_{\ell,N}'\ElN}{N(2\ell+1)}+ \lambda \norm{\phiLASSO}_1\leq \lambda \norm{\phib_\ell}_1\,.
	$$
\end{proof}

\begin{proof}[Proof of Proposition \ref{DevCond}]
	First of all, we have
	$$
	\norm{\widehat\gamma_{\ell,N}-\widehat\Gamma_{\ell,N}\phib_\ell}_{\infty}=\frac 1{N(2\ell+1)}\norm{X_{\ell,N}'\ElN}_{\infty}=\max_{1\leq h \leq p} \abs{\frac{\mathbf{Y}^{\prime}_{\ell, N} (h)\ElN}{N(2\ell+1)}}\,.
	$$
	Now, using \eqref{eq:yeDB}, we obtain that, for any $\eta\geq0$ and $c>0$,
	\begin{equation*}
	\P \( \abs{\frac{\mathbf{Y}^{\prime}_{\ell, N} (h)\ElN}{N(2\ell+1)}} >   C_{\ell;Z} \( 1+ \frac{1+ \mu_{max; \ell}}{\mu_{\min; \ell}}\)\eta \) \le 6 \exp \( -c\, N(2\ell+1)\, \min\{\eta, \eta^2\}\)\,.
	\end{equation*}
	Thus it follows that
	\begin{equation*}
	\P \( \max_{1\leq h \leq p}\abs{\frac{\mathbf{Y}^{\prime}_{\ell, N} (h)\ElN}{N(2\ell+1)}} >   C_{\ell;Z} \( 1+ \frac{1+ \mu_{max; \ell}}{\mu_{\min; \ell}}\)\eta \) \\
	\le  6\, p \,\exp \( -c\, N(2\ell+1)\, \min\{\eta, \eta^2\}\)\,.
	\end{equation*}
	Since, for every $\ell < L_N$,
	$$
	c_0 C_{\ell;Z} \( 1+ \frac{1+ \mu_{max; \ell}}{\mu_{\min; \ell}} \) \le \mathcal{F}_N,
	$$
	we have 
	$$
	\P \( \max_{\ell < L_N}\max_{1\leq h \leq p}\abs{\frac{\mathbf{Y}^{\prime}_{\ell, N} (h)\ElN}{N(2\ell+1)}} >  \frac{1}{c_0} \mathcal{F}_N \eta \) \\
	\le  6\, p \, L_N  \,\exp \( -c\, N \, \min\{\eta, \eta^2\}\)\,.
	$$
	Hence, for $\eta = c_0 \sqrt{\frac{\log p L_N }{N}}$ 
	\begin{flalign*}
	&\P\(\bigcap_{\ell = 0}^{L_N -1} \norm{\widehat\gamma_{\ell,N}-\widehat\Gamma_{\ell,N}\phib_\ell}_{\infty} \leq  \mathcal{F} _N \sqrt{\frac{\log p L_N }{N}} \)\\
	&\qqqquad\qqqquad\qqqquad=\P\(\max_{\ell < L_N} \norm{\widehat\gamma_{\ell,N}-\widehat\Gamma_{\ell,N}\phib_\ell}_{\infty} \leq  \mathcal F _N \sqrt{\frac{\log p L_N }{N}} \)\\
		&\qqqquad\qqqquad\qqqquad=\P \( \max_{\ell < L_N}  \max_{1\leq h \leq p} \abs{\mathbf{Y}_{\ell,N}(h)'\ElN} \leq   \mathcal F _N \sqrt{\frac{\log p L_N }{N}} \) \\
		&\qqqquad\qqqquad\qqqquad\ge 1-6\,p \, L_N \,\exp \( -c\,N \,\min\{c_0, c_0^2\}\,\frac{\log pL_N}{N}\)\\
		&\qqqquad\qqqquad\qqqquad= 1-6\,p \, L_N \, e^{ -c\,\min\{c_0, c_0^2\}\,\log p L_N}\\
		&\qqqquad\qqqquad\qqqquad= 1-6\,e^{-\(c\,\min\{c_0, c_0^2\}-1\)\log p L_N}\,,
	\end{flalign*}
	and the statement is proved with $c_1=6$, $c_2=c\,\min\{c_0, c_0^2\}-1$, where $c_0$ is any positive constant that satisfies $c_2>0$.
\end{proof}

\begin{proof}[Proof of Proposition \ref{RE_thm}]
	Our goal is to prove \eqref{eq:RE_thm}, which can be rewritten as follows
	$$
	\P \( \bigcap_{\ell=0 }^{L_N -1} A_\ell \) \ge 1-c_1\,e^{-c_2\,N \,\min\{\omega^{-2},1\}}\,,
	$$
	where $A_\ell = \{ v_\ell' \widehat\Gamma_{\ell,N} v_\ell \ge \alpha_\ell\norm{v_\ell}_2^2-\tau_\ell \norm{v_\ell }_1^2, \ \forall \, v_\ell \in \mathbb{R}^p \}$.
	
	We start from Equation \eqref{eq:vvDB},
	and considering that
	$$2\pi\mathcal{M}(\widetilde f_\ell, r )\le 2\pi\mathcal{M}(\widetilde f_\ell)\le p\,2\pi\mathcal{M}(f_\ell)\le p\,\frac{C_{\ell;Z}}{\mu_{\min; \ell}}, \quad \rr \ge 1,$$ we have 
	$$
	\P \(\abs{ v_\ell' \( \widehat\Gamma_{\ell,N} - \Gamma_\ell \) v_\ell } > \frac{p\,C_{\ell;Z}}{\mu_{\min; \ell}} \eta \) \le 2e^{-c\,N(2\ell+1)\, \min\{ \eta^2, \eta\}} \,.
	$$
	Using Lemma F.2 in the supplementary material of \cite{BM:15} yields
	\begin{flalign*}
	&\P \(\sup_{v_\ell\in \mathscr K(2s)} \abs{v_\ell ' \( \widehat\Gamma_{\ell,N} - \Gamma_\ell \) v_\ell }>\eta\,\frac{p\,C_{\ell;Z}}{\mu_{\min; \ell}}\)\leq 2e^{-cN(2\ell+1)\min\{ \eta^2, \eta\}+2s\min\{\log p,\log\(\frac{21\,e\,p}{2s}\)\}}\,,
	\end{flalign*}
	where $\mathscr K(2s) = \{ v \in \mathbb{R}^p: \norm{v}_2 \le 1, \, \norm{v}_0 \le 2s \}$, for an integer $s \ge 1$.
	Now we set 
	$$
	c_3=54  \quad \eta=\omega_\ell^{-1}=\frac{1}{54\,p}\frac{\mu_{\min; \ell}}{\mu_{\max; \ell}},
	$$
	to obtain
	\begin{flalign*}
	&\P \(\sup_{v_\ell \in \mathscr K(2s)} \abs{v_\ell' \( \widehat\Gamma_{\ell,N} - \Gamma_\ell \) v_\ell} \le \frac1{54}\,\frac{C_{\ell;Z}}{\mu_{\max; \ell}}\)\ge 1- 2e^{-cN(2\ell+1)\min\{ 1, \omega_\ell^{-2}\}+2s\min\{\log p,\log\(\frac{21\,e\,p}{2s}\)\}},
	\end{flalign*}
	and apply Lemma 12 of the supplementary material of \cite{LW:12} to get
	\begin{align*}
	&\P \(\abs{v_\ell' \( \widehat\Gamma_{\ell,N} - \Gamma_\ell \) v_\ell} \le \frac12\,\frac{C_{\ell;Z}}{\mu_{\max; \ell}}\{\norm{v_\ell}_2^2+\frac1s\norm{v_\ell}_1^2\}, \ \forall \, v_\ell \in \mathbb{R}^p\)
	\\& \qqqquad\qqqquad\qqqquad\ge 1- 2e^{-cN(2\ell+1)\min\{ 1, \omega_\ell^{-2}\}+2s\min\{\log p,\,\log\(\frac{21\,e\,p}{2s}\)\}}\,.
	\end{align*}	
Moreover, it holds that
	\begin{flalign*}
	\abs{v_\ell' \( \widehat\Gamma_{\ell,N} - \Gamma_\ell \) v_\ell}&=\abs{v_\ell '\widehat\Gamma_{\ell,N} v_\ell - v_\ell'\Gamma_\ell v_\ell}=\abs{v_\ell'\Gamma_\ell v_\ell-v_\ell'\widehat\Gamma_{\ell,N} v_\ell}\geq\abs{v_\ell'\Gamma_\ell v_\ell}-\abs{v_\ell'\widehat\Gamma_{\ell,N} v_\ell}\\
	&\geq\Lambda_{\min}\(\Gamma_\ell\)\norm{v_\ell}_2^2-\abs{v_\ell'\widehat\Gamma_{\ell,N} v_\ell}\geq\frac{C_{\ell;Z}}{\mu_{\max;\ell}}\norm{v_\ell}_2^2-\abs{v_\ell'\widehat\Gamma_{\ell,N} v_\ell} \\&
	=\frac{C_{\ell;Z}}{\mu_{\max;\ell}}\norm{v_\ell}_2^2-v_\ell'\widehat\Gamma_{\ell,N} v_\ell\,,
	\end{flalign*}
	which implies that 
	\begin{align*}
	&\P \(\frac{C_{\ell;Z}}{\mu_{\max;\ell}}\norm{v_\ell}_2^2-v_\ell'\widehat\Gamma_{\ell,N} v_\ell \le \frac12\,\frac{C_{\ell;Z}}{\mu_{\max;\ell}}\{\norm{v_\ell}_2^2+\frac1s\norm{v_\ell}_1^2\} , \ \forall \, v_\ell \in \mathbb{R}^p \)\\& \qqqquad\qqqquad\qqqquad
	\ge 1- 2e^{-cN(2\ell+1)\min\{ 1, \omega_\ell^{-2}\}+2s\min\{\log p,\,\log\(\frac{21\,e\,p}{2s}\)\}}\,.
	\end{align*}
	Hence, we can rearrange the terms in the previous relation to have
	\begin{align*}
	&\P \(v_\ell'\widehat\Gamma_{\ell,N} v_\ell \ge \frac12\,\frac{C_{\ell;Z}}{\mu_{\max;\ell}}\norm{v_\ell}_2^2-\frac1{2s}\,\frac{C_{\ell;Z}}{\mu_{\max;\ell}}\norm{v_\ell}_1^2, \  \forall\, v_\ell \in \mathbb{R}^p\)
	\\& \qqqquad\qqqquad\qqqquad\ge 1- 2e^{-cN(2\ell+1)\min\{ 1, \omega_\ell^{-2}\}+2s\min\{\log p,\,\log\(\frac{21\,e\,p}{2s}\)\}}\, .
	\end{align*}
	Now, taking $\omega_N = \max_{\ell < \LN}  \omega_\ell $,
	we obtain
	\begin{flalign*}
	\P \( \bigcup_{\ell=0 }^{L_N -1} \overline{A}_\ell \) 
	& \le \sum_{\ell=0}^{L_N-1} \[ 1- \P \( A_\ell  \)\] \\
	&\le  2e^{-cN\min\{ 1, \omega_N^{-2}\}+2s\min\{\log pL_N,\,\log\(\frac{21\,e\,p L_N }{2s}\)\}}\,,
	\end{flalign*}
	which is the desired conclusion, once we set 
	$$
	s=\frac{c\,N\,\min\{\omega_N^{-2},1\}}{4\log pL_N}\, .
	$$ 
\end{proof}



\begin{proof}[Proof of Theorem \ref{th:MR-preciso}]
	Set $v_{\ell}=\phiLASSO-\phib_\ell$. Using the basic inequality \eqref{eq:BI} and the deviation condition \eqref{DB}, we obtain that, almost surely, 
	\begin{flalign}
	v_{\ell}' \widehat\Gamma_{\ell,N} v_{\ell} &\le 2v_{\ell}'\(\widehat\gamma_{\ell,N}- \widehat\Gamma_{\ell,N}\phib_\ell\)+\lambda_{N}\(\norm{\phib_\ell}_1-\norm{\phib_\ell+v_{\ell}}_1\)\notag\\
	&\le 2\norm{v_{\ell}}_1\norm{\widehat\gamma_{\ell,N}- \widehat\Gamma_{\ell,N}\phib_\ell}_{\infty}+\lambda_{N}\(\norm{\phib_\ell}_1-\norm{\phib_\ell+v_{\ell}}_1\)\notag\\
	&\le 2\norm{v_{\ell}}_1 \mathcal{F}_N\, \sqrt{\frac{\log (pL_N)}{N}}+\lambda_{N}\(\norm{\phib_\ell}_1-\norm{\phib_\ell+v_{\ell}}_1\)\,.\label{bound1}
	\end{flalign}
	Now, let $J=\text{supp}(\phib_\ell)=\{j_1,\dots,j_{\ql}\}$ be such that $\abs{J}=\ql$, then $J^c=\{1,\dots,p\}\setminus J $, $\norm{\phib_{\ell,J}}_1=\norm{\phib_\ell}_1$ and $\norm{\phib_{\ell,J^c}}_1=0$. Consequently, it holds that
	\begin{flalign*}
	\norm{\phib_\ell+v_{\ell}}_1&= \norm{\phib_{\ell,J}+v_{\ell,J}}_1+\norm{v_{\ell,J^c}}_1\\&\ge  \norm{\phib_{\ell,J}}_1-\norm{v_{\ell,J}}_1+\norm{v_{\ell,J^c}}_1\,,
	\end{flalign*}
	which implies 
	\begin{flalign*}
	\lambda_{N}\(\norm{\phib_\ell}_1-\norm{\phib_\ell+v_{\ell}}_1\)&\le \lambda_{N}\(\norm{\phib_{\ell,J}}_1- \norm{\phib_{\ell,J}}_1+\norm{v_{\ell,J}}_1-\norm{v_{\ell,J^c}}_1\)\le \lambda_{N}\(\norm{v_{\ell,J}}_1-\norm{v_{\ell,J^c}}_1\)\,.
	\end{flalign*}

	Having explicitly required that $\lambda_{N} \ge4  \mathcal{F}_N\, \sqrt{\log (pL_N)/N}$, Equation \eqref{bound1} becomes
	\begin{flalign}\label{bound2}
	0\le v_{\ell}' \widehat\Gamma_{\ell,N} v_{\ell} &\le \frac{\lambda_{N}}{2}\norm{v_{\ell}}_1 +\lambda_{N}\(\norm{v_{\ell,J}}_1-\norm{v_{\ell,J^c}}_1\) \notag\\
	&=\frac{\lambda_{N}}{2}\(\norm{v_{\ell,J}}_1+\norm{v_{\ell,J^c}}_1\) +\lambda_{N}\(\norm{v_{\ell,J}}_1-\norm{v_{\ell,J^c}}_1\) \notag\\
	&=\frac{3\lambda_{N}}{2}\norm{v_{\ell,J}}_1-\frac{\lambda_{N}}2 \norm{v_{\ell,J^c}}_1 \le \frac{3}{2} \lambda_{N} \norm{v_{\ell}}_1\,.
	\end{flalign}
	This ensures that $\norm{v_{\ell,J^c}}_1\le 3\,\norm{v_{\ell,J}}_1$ and hence, adding $\norm{v_{\ell,J}}_1$ on both sides, that $\norm{v_{\ell}}_1\le 4\,\norm{v_{\ell,J}}_1$, which implies that
	$$
	\norm{v_{\ell}}_1\le 4\,\sqrt{\ql} \,\norm{v_{\ell}}_2\,,
	$$
	from Cauchy-Schwartz inequality.
	
	Now we use this property into the (RE) inequality \eqref{RE}, keeping in mind that we specifically required that $\ql \tau_\ell \le\alpha_\ell/32$, and we obtain
	\begin{equation*}
	v_{\ell}' \widehat\Gamma_{\ell,N} v_{\ell} \ge \alpha_\ell \norm{v_{\ell}}_2^2 - \tau_\ell \norm{v_{\ell}}_1^2 \ge  \alpha_\ell \norm{v_{\ell}}_2^2 - 16 \ql \tau_\ell \norm{v_{\ell}}_2^2 
	\end{equation*}
	\begin{equation}\label{bound3}
	\ge \alpha_\ell  \norm{v_{\ell}}_2^2 - \frac{\alpha_\ell}{2}  \norm{v_{\ell}}_2^2 \ge \frac{\alpha_\ell}{2}  \norm{v_{\ell}}_2^2\,.
	\end{equation}
	Hence, combining Equation \eqref{bound2} and \eqref{bound3}, we get 
	\begin{flalign*}
	\frac{\alpha_\ell}{2}  \norm{v_{\ell}}_2^2 \le v_{\ell}' \widehat\Gamma_{\ell,N} v_{\ell}  \le \frac{3}{2} \lambda_{N} \norm{v_{\ell}}_1 \le 6 \sqrt \ql \,\lambda_N \norm{v_{\ell}}_2\,,
	\end{flalign*}
	which results in the following estimate for the norm of the error
	\begin{flalign*}
	\frac{\alpha_\ell}{3}  \norm{v_{\ell}}_2^2 \le \lambda_{N} \norm{v_{\ell}}_1 \le \lambda_N  4 \sqrt {\ql} \,\norm{v_{\ell}}_2\,.
	\end{flalign*}
	As a consequence,
	\begin{flalign}
	&\norm{v_{\ell}}_2 \le 12 \sqrt {\ql} \, \frac{\lambda_{N}}{\alpha_\ell}\,, \label{EstPredError}\\
	&\norm{v_{\ell}}_1  \le 4 \sqrt {\ql} \,\norm{v_{\ell}}_2 \le 48\,\ql \, \frac{\lambda_{N}}{\alpha_\ell} \notag \,,\\
	&v_{\ell}' \widehat\Gamma_{\ell,N} v_{\ell}  \le \frac{3}{2} \lambda_{N} \norm{v_{\ell}}_1 \le 72\, \frac{\lambda_{N}^2}{\alpha_\ell}\,.\notag
	\end{flalign}

	
	It is readily seen that	
	\begin{align*}
	\norm{\widehat{\mathbf{k}}^{\textnormal{lasso}}_N-\mathbf{k}}^2_{L^2} &= \norm{\widehat{\mathbf{k}}^{\textnormal{lasso}}_N-\mathbf{k}_N}^2_{L^2}+\Big \| \ka - \ka_{N} \Big \|^2_{L^2}\\
	&=   \sum_{\ell =0}^{L_N} \sum_{\ell' =0}^{L_N} \left \langle \phiLASSO - \boldsymbol{\phi}_\ell, \phiLASSOpr - \boldsymbol{\phi}_{\ell'}\right \rangle    \int_{-1}^1 \frac{2\ell +1}{4\pi} P_\ell (z)   \frac{2\ell' +1}{4\pi} P_{\ell'} (z) \diff z +\Big \| \ka - \ka_{N} \Big \|^2_{L^2} \\
	&= \sum_{\ell = 0}^{\LL} \left \| \phiLASSO - \boldsymbol{\phi}_\ell \right\|_2^2 \frac{2\ell +1}{8\pi^2}  +\Big \| \ka - \ka_{N} \Big \|^2_{L^2}\,,
\end{align*}
by the orthonormality of Legendre polynomials (see \eqref{eq:duplication}), and that
	\begin{equation*}
	\norm{\widehat{\mathbf{k}}^{\textnormal{lasso}}_N-\mathbf{k}}_{L^\infty}  \le \sum_{\ell =0}^{\LL} \left \| \phiLASSO - \boldsymbol{\phi}_\ell \right \|_2\frac{2\ell +1}{4\pi} +  \Big \| \ka - \ka_{N} \Big \|_{L^\infty},
	\end{equation*}
	using the triangle inequality. Recalling that $\phiLASSO - \boldsymbol{\phi}_\ell=v_\ell$, by \eqref{EstPredError}, the proof is concluded.
	
\end{proof}
\section*{Acknowledgement} 
The authors wish to thank Domenico Marinucci for many insightful discussions and suggestions. AV acknowledges the MIUR Excellence Department Project awarded to the Department of Mathematics, University of Rome Tor Vergata, CUP E83C18000100006.

\bibliography{Bibliography}
\end{document}